\documentclass{amsart}

\usepackage{amssymb}
\usepackage{amsmath}
\usepackage{color}

\newcommand\NN{{\mathbb N}}
\newcommand\RR{{{\mathbb R}}}

\newcommand\cA{{\mathcal A}}

\newcommand\cC{{\mathcal C}}
\newcommand\cD{{\mathcal D}}

\newcommand\cT{{\mathcal T}}

\newtheorem{theo}{Theorem}[section]
\newtheorem{lemm}[theo]{Lemma}

\newtheorem{prop}[theo]{Proposition}
\newtheorem{rema}[theo]{Remark}

\begin{document}

\title[Prandtl equation]
{Well-posedness of The Prandtl Equation in Sobolev Spaces}

\author{R. Alexandre}
\address{Radjesvarane Alexandre
\newline\indent
Department of Mathematics, Shanghai Jiao Tong University
\newline\indent
Shanghai, 200240, P. R. China
\newline\indent
and 
\newline\indent
Ecole Navale, Brest Lanv\'eoc Poulmic 29200, France }
\email{radjesvarane.alexandre@ecole-navale.fr}

\author{Y.-G. Wang}
\address{Ya-Guang Wang
\newline\indent
Department of Mathematics, Shanghai Jiao Tong University
\newline\indent
Shanghai, 200240, P. R. China}
\email{ygwang@sjtu.edu.cn}

\author{C.-J. Xu}
\address{Chao-Jiang Xu
\newline\indent
School of Mathematics, Wuhan University 430072,
Wuhan, P. R. China
\newline\indent
and \newline\indent
Universit\'e de Rouen, UMR 6085-CNRS,
Math\'ematiques
\newline\indent
Avenue de l'Universit\'e,\,\, BP.12, 76801 Saint
Etienne du Rouvray, France } \email{Chao-Jiang.Xu@univ-rouen.fr}

\author{T. Yang}
\address{Tong Yang
\newline\indent
Department of mathematics, City University of Hong Kong,
Hong Kong, P. R. China} \email{matyang@cityu.edu.hk}

\subjclass[2000]{35M13, 35Q35, 76D10, 76D03, 76N20}

\date{}

\keywords{Prandtl equation,  well-posedness theory, Sobolev spaces, energy method,
monotonic velocity field, Nash-Moser-H\"ormander iteration.}

\begin{abstract}
We develop a new approach to study the well-posedness theory of the Prandtl equation in
Sobolev spaces by using a {\it direct energy method} under a
monotonicity condition on the tangential velocity field
instead of using the Crocco transformation. Precisely, we firstly
investigate the linearized Prandtl equation
in some weighted Sobolev spaces when the tangential velocity of the  background state
is  monotonic in the normal variable. Then to cope with the loss of regularity
of the perturbation with respect to the background state due to the degeneracy of
the equation,  we apply  the Nash-Moser-H\"ormander iteration to obtain a well-posedness theory
 of classical solutions to the nonlinear Prandtl equation  when the initial data
 is a small perturbation of a monotonic shear flow.
\end{abstract}

\maketitle

\tableofcontents

\section{Introduction}\label{s1}

In this work, we study the well-posedness of the Prandtl equation which is the foundation
of the boundary layer theory introduced by Prandtl in 1904,  \cite{prandtl}.
It describes the behavior of the flow near a boundary  in the small viscosity limit
of an incompressible viscous flow with the non-slip boundary condition.
We consider the following initial-boundary value
problem,
\begin{equation}\label{prandtl}
\left\{\begin{array}{l} u_t + u u_x + vu_y + p_x
= u_{yy},\quad t>0,\quad x\in\RR,\quad y>0, \\
u_x +v_y =0, \\
u|_{y=0} = v|_{y=0} =0 , \ \lim\limits_{y\to+\infty} u = U(t,x), \\
u|_{t=0} =u_0 (x,y)\, , \end{array}\right.
\end{equation}
where $u(t,x,y)$ and $v(t,x,y)$ represent the tangential and normal
velocities of the boundary layer, with $y$ being the scaled normal variable
to the boundary, while $U(t,x)$ and $p(t,x)$ are the values on the
boundary of the tangential velocity and pressure of the outflow satisfying the Bernoulli law
\[
\partial_t U + U\partial_x U +\partial_x p=0.
\]

The well-posedness theories and the justification of the Prandtl equation remain as the challenging
 problems in the mathematical theory of fluid mechanics. Up to now, there are only  a few  rigorous
mathematical results.
 Under a monotonic assumption on the tangential
velocity of the outflow, Oleinik was the first to obtain
the local existence of classical solutions for the initial-boundary value problems
 in \cite{oleinik-1}, and this result together with some of her
works with collaborators were well presented in
the monograph \cite{oleinik-3}.  In addition
to Oleinik's monotonicity assumption on the velocity field, by imposing
 a so called favorable condition on the pressure, Xin and Zhang obtained the existence of global
 weak solutions to the Prandtl equation in \cite{xin-zhang}. All these well-posedness results were
 based on the Crocco transformation to overcome  the main difficulty caused by  degeneracy and
 mixed type of the equation.

Without the monotonicity assumption,  E and Engquist in \cite{e-2} constructed some finite time blowup
solutions to the Prandtl equation. And
in \cite{Samm}, Sammartino and Caflisch obtained the
local existence of analytic solutions to the Prandtl equation, and
a rigorous theory on the stability of boundary layers  with analytic data in the framework of the
abstract Cauchy-Kowaleskaya theory. This result was
extended to the function space which is
only analytic  in the tangential variable in \cite{cannone}.


In recent years, there have been some interesting works concerning the linear and nonlinear
instability of the Prandtl equation in the Sobolev spaces. In \cite{grenier}, Grenier showed
that the unstable Euler shear flow $(u_s(y), 0)$ with $u_s(y)$ having an inflection point
(the well-known Rayleigh's criterion) yields instability for the Prandtl equation. In the spirit of Grenier's approach, in \cite{GV-D}, G\'erard-Varet and Dormy
showed that if the shear flow profile $(u^s(t,y), 0)$ of the Prandtl equation has a non-degenerate
critical point, then it leads to a strong linear ill-posedness of the Prandtl equation in the
Sobolev framework. In a similar approach, in \cite{GV-N} G\'erard-Varet and Nguyen strengthened
the ill-posedness result of \cite{GV-D} for the linearized Prandtl equation for an unstable
shear flow. Moreover,
they also showed that if a solution, as a small perturbation of the unstable shear
flow, to the nonlinear Prandtl equation exists in
the Sobolev setting, then it cannot be Lipschitz continuous.
Along this direction, Guo and Nguyen in \cite{guo} proved that
the nonlinear Prandtl equation is ill-posed near non-stationary and non-monotonic shear flows, and  showed that
the asymptotic boundary-layer
expansion is not valid for non-monotonic shear layer flows in Sobolev spaces.
Hong and Hunter \cite{hong-hunter} studied the formation
 of singularities and instability in the unsteady inviscid
Prandtl equation. All these works show what happens in an unsteady boundary
layer separation. For the related mathematical
results and discussions, also see the review papers  \cite{caf, e-1}.

As mentioned in \cite{caf, GV-D, GV-N}, the local in time well-posedness of the Prandtl equation for
initial data in Sobolev space remains an open problem. Therefore, it would be very interesting
to recover  Oleinik's well-posedness results simply through direct energy estimates. And
this is the goal of this paper.

In this paper we consider the case of an uniform outflow $U =1$, which implies $p$ being a constant.
Consider the following problem for the Prandtl equation,
\begin{equation}\label{eq-1}
\left\{\begin{array}{l} u_t + u u_x + vu_y - u_{yy}=0,
\quad t>0,\quad x\in\RR,\quad y>0, \\
u_x +v_y =0, \\
u|_{y=0} = v|_{y=0} =0 , \ \lim\limits_{y\to +\infty}u = 1, \\
u|_{t=0} =u_0 (x,y)\, . \end{array}\right.
\end{equation}

We are going to study the well-posedness of the Prandtl equation around a monotonic
shear flow. More precisely, assume that
$$
u_0(x, y)=u_0^s(y)+\tilde{u}_0(x, y),
$$
where $u_0^s$ is  monotonic in $y$
\begin{equation}\label{oleinik-b}
\partial_y u_0^s (y)>0,\quad \forall y\geq 0.
\end{equation}

To state the main result, we first introduce the following
notations. Set $\Omega_T=[0, T]\times \RR^2_+$. For any non-negative integer
$k$ and real number $\ell$,  define
$$
\|u\|_{\mathcal{A}^{k}_{\ell}(\Omega_T)}=\Big(\sum_{k_1+
\left[\frac{k_2+1}{2}\right]\leq k}\|\langle y
\rangle^\ell\partial^{k_1}_{(t,x)}\partial^{k_2}_y u
\|^2_{L^2([0, T]\times\RR^2_+)}\Big)^{1/2}\,,
$$
and
$$\|u\|_{{\mathcal D}_{\ell}^k(\Omega_T)}=
\sum\limits_{k_1+[\frac{k_2+1}{2}]\le k}\|\langle y
\rangle^\ell\partial_{(t,x)}^{k_1}
\partial_y^{k_2}u\|_{L^\infty_{y}(L^2_{t,x})},$$
with $\langle y \rangle=(1+y^2)^{\frac 12}$.
When the function is independent of $t$ (or $x$) variable, we use the same
notations for the non-isotropic norms as above under the convention
that we do not take integration with respect to this variable.

The main result of this paper can be stated as follows.

\begin{theo}\label{theo1} Concerning the problem \eqref{eq-1}, we have the following
existence and stability results.

(1) Given any integer $k\ge 5$ and real number $\ell>\frac{1}{2}$, let the initial
data $u_0(x,y)=u_0^s(y)+\tilde{u}_0(x,y)$ satisfy the compatibility conditions of the
initial boundary value problem \eqref{eq-1} up to order $k+4$. Assume
the following two conditions:

(i) The monotonicity condition \eqref{oleinik-b} holds for $u_0^s$, and
 \begin{equation}\label{cond-1}
 \begin{cases}
 \partial_y^{2j}u_0^s(0)=0, \qquad \forall~0\le j\le k+4,\\
\|\langle y \rangle^\ell (u_0^s-1)\|_{{H}^{2k+9}(\RR_+^2)}+\|\frac{\partial^2_y u^s_0}{\partial_y u^s_0}\|_{{H}^{2k+7}(\RR_+^2)}\le C\,,
\end{cases}
\end{equation}
for a fixed constant $C>0$.

(ii) There exists a small constant $\epsilon>0$ depending only on  $u_0^s$, such that
\begin{equation}\label{cond-2}
\|\tilde{u}_0\|_{\cA^{2k+9}_{\ell}(\RR_+^2)}+
\|\frac{\partial_y\tilde{u}_0}{\partial_yu^s_0}\|
_{\cA^{2k+9}_{\ell}(\RR_+^2)}\le \epsilon.
\end{equation}

Then there is $T>0$, such that the problem \eqref{eq-1} admits a classical solution $u$ satisfying
\begin{equation}\label{space}
  u- u^s \in  \cA_\ell^{k}(\Omega_T), \,\, \frac{\partial_y(u-u^s)}{\partial_yu_0^s}\in
  \cA_\ell^{k}(\Omega_T), \,\, v\in \cD_0^{k-1}(\Omega_T).
\end{equation}
Here, note that $(u^s(t,y), 0)$ is the shear flow of the Prandtl equation defined by  the initial data $(u_0^s(y), 0)$.

(2) The solution is unique in the function space described in \eqref{space}. Moreover,
we have  the stability with respect to the initial data in the following sense.
For any given two initial data
$$
u_0^1=u_0^s+\tilde{u}^1_0,\quad u_0^2=u_0^s+\tilde{u}^2_0,
$$
if $u_0^s$ satisfy  \eqref{oleinik-b}, \eqref{cond-1}, and $\tilde{u}^1_0, \tilde{u}^2_0$
satisfy \eqref{cond-2}, then the corresponding solutions $(u^1, v^1)$ and $(u^2, v^2)$ of  \eqref{eq-1} satisfy
\begin{equation*}
  \|u^1-u^2\|_{\cA_\ell^{p}(\Omega_T)}+\|v^1-v^2\|_{ \cD_0^{p-1}(\Omega_T)}\le C\|\frac{\partial }{\partial y}\left(\frac{u^1_0-u_0^2}{\partial_y u_0^s}\right)\|_{\cA_\ell^p(\RR^2_+)},
\end{equation*}
for all $p\le k-1$, where the constant $C>0$ depends only on $T$ and the upper bounds of the norms of $u^1_0, u_0^2$.
\end{theo}

\begin{rema}
Note that the solutions obtained in the above theorem are
less regular than  the initial data mainly due to the  degeneracy  with respect to the tangential variable $x$
of the Prandtl equation.
\end{rema}

\begin{rema}
(1) It is not difficulty to see from the proof of Theorem \ref{theo1} that the above main result holds
also for the problem \eqref{eq-1} defined in the torus $\mathbb{T}^1$ for the $x-$variable.

(2) It can be seen by using our approach that the above well-posedness result can be generalized to
the nonlinear problem \eqref{prandtl} with a non-trivial Euler outflow. For example, for a smooth
positive tangential velocity $U(t,x)$ of the outflow, if the monotonic initial data $u_0(x,y)$
converges to $U(0,x)$ exponentially fast as $y\to +\infty$, and there is $\alpha>0$ such that
$u_0(x,y)-U(0,x)(1-e^{-\alpha y})$ is small in some weighted Sobolev spaces, then a similar local well-posedness
result holds for the nonlinear Prandtl equation \eqref{prandtl}, with the role of $u_0^s$ being replaced
by $U(0,x)(1-e^{-\alpha y})$.
\end{rema}

The rest of the paper is organized as follows. In  Section \ref{s2}, we will introduce some weighted non-isotropic
Sobolev spaces which will be used later, and give the properties
of the monotonic shear flow produced by the initial data $(u_0^s(y), 0)$. Then, in  Section \ref{s3}, we will study the
well-posedness of the linearized problem of the Prandtl equation (\ref{eq-1}) in the Sobolev spaces by a
direct energy method, when the background tangential velocity is monotonic in the normal variable. Again,
note that without using the Crocco transformation, the approach introduced here is completely new and will
have further applications. The proof of
Theorem \ref{theo1} will be presented in  Section \ref{section4},  Section \ref{section5} and  Section
\ref{section6}. As we mentioned earlier, from the energy
estimates obtained in  Section \ref{s3} for the linearized Prandtl equation, there is a loss of regularity of solutions
with respect to the background state and the initial data. For this, we apply the Nash-Moser-H\"ormander
iteration scheme to study the nonlinear Prandtl equation. In  Section \ref{section4}, we will first construct the iteration
scheme for the  problem \eqref{eq-1}, then in  Section \ref{section5} we will prove the convergence of this iteration scheme
by a series of estimates and then conclude the existence of solutions to the Prandtl equation. The uniqueness
and stability of the solution will be proved in  Section \ref{section6}. Finally,  Section \ref{section7} is devoted to the proof of several
technical estimates used in  Section \ref{section5}.

\bigskip
\section{Preliminary}\label{s2}
\setcounter{equation}{0}

As a preparation, in this section we introduce some functions spaces which will be used later. And the properties
of a monotonic shear flow will also be given.

\subsection{Weighted non-isotropic Sobolev spaces}\label{s2.1}

Since the Prandtl equation given in \eqref{eq-1} is a degenerate parabolic equation coupled with
the divergence-free condition, it is natural to work in some weighted anisotropic Sobolev spaces.

In addition to the spaces $\cA_\ell^{k}$ and $ \cD_\ell^{k}$ introduced already in Section \ref{s1}, we also need
the following function spaces.

Denote by $\partial^k_{\mathcal{T}}$ the
summation of tangential derivatives $\partial^\beta_{\mathcal{T}}=
\partial_t^{\beta_0}\partial_x^{\beta_1}$ for all
$\beta=(\beta_0, \beta_1)\in\NN^2$ with $|\beta|\leq k$. Recall  $\langle y\rangle=(1+|y|^2)^{1/2}$.
For any given $k, k_1, k_2\in\NN, \lambda\ge 0, \ell\in\RR$ and $0<T<+\infty$, we introduce
$$
\|f\|_{\mathcal{B}^{k_1, k_2}_{\lambda, \ell}}=
\Big(\sum_{0\le m\leq k_1, 0\leq q\leq k_2}\|e^{-\lambda t}
\langle y\rangle^\ell\partial^m_{\mathcal{T}}\partial^q_y
f\|^2_{L^2([0, T]\times\RR^2_+)}\Big)^{1/2}\, ,
$$
$$
\|f\|_{\tilde{\mathcal{B}}^{k_1, k_2}_{\lambda, \ell}}=
\Big(\sum_{0\le m\leq k_1, 0\leq q\leq k_2}\|e^{-\lambda t}
\langle y\rangle^\ell\partial^m_{\mathcal{T}}\partial^q_y f
\|^2_{L^\infty([0, T]; L^2(\RR^2_+))}\Big)^{1/2}\,,
$$
and $$
\|u\|_{{\mathcal C}_{\ell}^k}=
\sum\limits_{k_1+[\frac{k_2+1}{2}]\le k}\|\langle y\rangle^\ell\partial_{\mathcal T}^{k_1}
\partial_y^{k_2}u\|_{L^2_{y}(L^\infty_{t,x})}.$$
For the space $
\mathcal{A}^{k}_{\ell}$ introduced in  Section \ref{s1}, obviously, we have
\begin{equation*}
\mathcal{A}^{k}_{\ell}=\bigcap\limits_{j=0}^k
\mathcal{B}^{k-j, 2j}_{0,\ell}.
\end{equation*}

As mentioned in  Section \ref{s1}, when the function is independent of $t$ (or $x$) variable, we use
the same notations for the non-isotropic norms defined above under the convention
that we do not take integration or supremum with respect to this variable.
Note that the parameter $\lambda$ is associated to the variable $t$ and the
 parameter $\ell$ to  the
variable $y$.

The homogeneous norms $\|\,\cdot\,\|_{\dot{{\mathcal A}}_{\ell}^k},
\|\,\cdot\,\|_{\dot{{\mathcal C}}_{\ell}^k}, \|\,\cdot\,\|_{\dot{{\mathcal D}}_{\ell}^k}$
correspond to the summation $1\le k_1+[\frac{k_2+1}{2}]\le k$ in the definitions.

For $1\leq p\leq +\infty$, we will also use $\|f\|_{L^p_\ell(\RR^2_+)}
=\|\langle y\rangle^\ell f\|_{L^{^p}(\RR^2_+)}$.

By classical theory, it is easy to get the following Sobolev type embeddings
\begin{equation}\label{sob-C}
\|u\|_{{\mathcal C}_{\ell}^k}\leq C_s\|u\|_{{\mathcal A}_{\ell}^{k+2}},\qquad
\|u\|_{{\mathcal D}_{\ell}^k}\leq C_s\|u\|_{{\mathcal A}_{\ell}^{k+1}}\,.
\end{equation}
Moreover, for any $\ell\ge 0$ and $k\ge 2$, the space ${\mathcal A}_{\ell}^k$ is continuously embedded
into $C_b^{k-2}$, the space of $(k-2)-$th order continuously differentiable functions with all
derivatives being bounded.

In the following, we will use some Morse-type inequalities
for the above four function spaces, which are consequences
of interpolation inequality and the fact that the space
$L^2\cap L^\infty$ is an algebra.

\begin{lemm}\label{lemm2.1}
For any  proper functions $f$ and $g$, we have
$$
\|f\,g\|_{{\mathcal A}_{\ell}^k}\le M_k \{\|f\|_{{\mathcal A}_{\ell}^k}
\|g\|_{L^\infty}+\|f\|_{L^\infty}\,\|g\|_{\dot{\mathcal A}_{\ell}^k}\}\,,
$$
and
$$
\|f\,g\|_{{\mathcal A}_{\ell}^k}\le M_k \{\|f\|_{{\mathcal C}_{\ell}^k}\|g
\|_{{\mathcal D}_{0}^0}+\|f\|_{{\mathcal C}_{\ell}^0}
\|g\|_{\dot{\mathcal D}_{0}^k}\}\,.
$$
Similar inequalities hold for the norms  $\|\,\cdot\,\|_{{\mathcal B}_{\ell}^{k_1, k_2}},
\|\,\cdot\,\|_{{\mathcal C}_{\ell}^{k}}$ and $\|\,\cdot\,\|_{{\mathcal D}_{\ell}^{k}} $.
Here, $M_k>0$ is  a constant depending only on $k$.
\end{lemm}

This result can be obtained in a way similar to that given in \cite{metivier}.

\subsection{Properties of a monotonic shear flow} \label{s2.2-1}
Let $u^s (t,y)$ be the solution of the following initial boundary value problem
\begin{equation}\label{heat-1}
\left\{\begin{array}{l}
\partial_t u^s = \partial^2_y u^s ,\quad t>0,\quad y>0, \\
{u^s}|_{y=0} = 0, \quad \lim\limits_{y\to+\infty} u^s(t,y) = 1, \quad t>0, \\
{u^s}|_{t=0} =u_0^s (y),\quad y>0.
\end{array}\right..
\end{equation}

Note that $(u^s(t,y),0)$ is a shear flow for the Prandtl equation \eqref{eq-1}.

\begin{prop} For any fixed  $k\ge 2$, assume that the initial data $u_0^s(y)$ satisfies the monotonicity condition
\begin{equation*}
 \partial_y u_0^s (y)>0,\quad \forall y\geq 0\,,
\end{equation*}
and the compatibility conditions for \eqref{heat-1} up to order $k$, i.e.
\begin{equation*}
\lim_{y\rightarrow+\infty}u_0^s(y)= 1,\
\quad\partial_y^{2j}u_0^s(0)= 0,\quad \forall ~0\le j\le k.
\end{equation*}
Moreover, assume that
\begin{equation*}
\|u_0^s-1\|_{L^2({\RR}_+)}+\|u_0^s\|_{L^\infty({\RR}_+)}+\|u_0^s\|_{\dot{\cC}^{k}_\ell}+\|\frac{\partial^2_y u^s_0}{\partial_y u^s_0}\|_{{\cC}^{k-1}_{0}}\le C\,,
\end{equation*}
for a constant $C>0$, and a fixed $\ell>0$. Then, we have
\begin{equation}\label{mono}
\partial_y u^s (t, y)>0,\quad \forall~ t, y\geq 0,\end{equation}
for any fixed $T>0$, there is a constant $C(T)$ such that
\begin{equation}\label{result-u-s}
\|u^s\|_{L^\infty([0,T]\times {\RR}_+)}+\|u^s\|_{\dot{\cC}^{k}_\ell}+
\|\frac{\partial^2_y u^s}{\partial_y u^s}\|_{{\cC}^{k-1}_{0}}\le C(T).
\end{equation}
Moreover,  for a fixed $0<R_0<T$, there is a constant $C(T, R_0)>0$ such that
\begin{equation}\label{cond-2}
\max_{0\le {\bar y}\le R_0}
\|\frac{\partial_y u^s (t,y+\bar{y})}{\partial_y u^s (t,y)}\|_{{{\cC}^{k-1}_{0}}([0,T]\times \RR_y^+)}\le C(T, R_0),
\end{equation}
and
\begin{equation}\label{cond-3}
\max_{0\le {\bar t}\le R_0}
\|\frac{\partial_y u^s (t+\bar t,y)}{\partial_y u^s (t,y)}\|_{{{\cC}^{k-2}_{0}}([0,T-R_0]\times \RR_y^+)}\le C(T, R_0).
\end{equation}
\end{prop}

\begin{proof}
Obviously,  the solution of \eqref{heat-1} can be written as
\begin{align*}
u^s (t,y) &=\frac{1}{2\sqrt {\pi t}} \int^{+\infty}_{0} \Big(
e^{-\frac{(y-\tilde{y})^2}{4 t}}-e^{-\frac{(y+\tilde{y})^2}
{4 t}}\Big)  u_0^s (\tilde{y}) d\tilde{y}\\
&=\frac{1}{\sqrt {\pi}} \Big(\int^{+\infty}_{- {\frac{y}{2\sqrt t}}}
 e^{-\xi^2} u_0^s (2\sqrt t \xi +y) d\xi - \int^{+\infty}_{
  {\frac{y}{2\sqrt t}}}  e^{-\xi^2}u_0^s (2\sqrt t \xi -y)d\xi \Big)\, ,
\end{align*}
which gives
\begin{align*}
\partial_t u^s(t, y) =& \frac{1}{\sqrt {\pi t}} \Big(
\int^{+\infty}_{- {\frac{y}{2\sqrt t}}} {\xi}\, e^{-\xi^2}
(\partial_y u_0^s) (2\sqrt t \xi +y) d\xi  \\
&\qquad- \int^{+\infty}_{ {\frac{y}{2\sqrt t}}}{\xi}\,
e^{-\xi^2}(\partial_y u_0^s) (2\sqrt t \xi -y)d\xi \Big).
\end{align*}
By using $\partial_y^{2j}u_0^s(0)=0$ for $0\le j\le k$, it follows
\begin{align*}
\partial^p_y u^s(t, y) =& \frac{1}{\sqrt \pi} \Big( \int^{+\infty}_{-
{\frac{y}{2\sqrt t}}}  e^{-\xi^2} (\partial^p_yu_0^s) (2\sqrt t \xi+y) d\xi\\
 &\quad+ (-1)^{p+1}\int^{+\infty}_{ {\frac{y}{2\sqrt t}}}
 e^{-\xi^2}(\partial^p_yu_0^s) (2\sqrt t \xi -y)d\xi \Big),
\end{align*}
for all $1\le p\le 2k$, from which one immediately deduces the monotonicity property \eqref{mono}.

To estimate the last term given in \eqref{result-u-s},
denote
$$
\alpha(t, y)=\frac{\partial^2_y u^s }{\partial_y u^s}(t, y).
$$
Then,  from \eqref{heat-1}, we know that $\alpha(t,y)$ satisfies the following initial boundary value problem for the Burgers equation,
\begin{equation}\label{heat-2}
\left\{\begin{array}{l}
\partial_t \alpha = \partial^2_y \alpha+2\alpha\partial_y \alpha ,\\
{\alpha}|_{y=0} = 0, \quad t>0, \\
{\alpha}|_{t=0} =\alpha^0 (y) :=\frac{\partial^2_y u_0^s }
{\partial_y u_0^s}(y),\quad y>0.
\end{array}\right.
\end{equation}
It is easy to verify that the compatibility conditions of \eqref{heat-2} hold up to order $k-1$,
the estimates on $\|\alpha\|_{\cC_0^{k-1}}$ can be obtained by
standard energy method after an odd extension of the initial data to the whole $\RR$.

To prove \eqref{cond-2}, note that
\begin{equation}\label{indentity}
\frac{\partial_y u^s (t, y+\bar y)}{\partial_y u^s (t, y)}  = \frac{\partial_y u^s (t,y+\bar y) -
 \partial_y u^s(t, y)}{\partial_y u^s(t,y)}  +1,\end{equation}
and
$$\partial_y u^s (t,y+\bar y) -
 \partial_y u^s(t, y)= \int^{\bar y}_0 \partial^2_y u^s (t,y+z)d z.$$
Hence,  from \eqref{indentity}, we have by using \eqref{result-u-s} that
\begin{equation}\label{indentity-1}
\|\frac{\partial_y u^s (t, y+\bar y)}{\partial_y u^s (t, y)}\|_{{\cC}^{k-1}_{0}([0,T]\times \RR^+_y)}\le
1+ C\int_0^{\bar y}\|\frac{\partial_y u^s (t, y+z)}{\partial_y u^s (t, y)}\|_{{\cC}^{k-1}_{0}([0,T]\times \RR^+_y)}
dz.\end{equation}

By applying the Gronwall inequality to \eqref{indentity-1}, it follows
$$\|\frac{\partial_y u^s (t, y+\bar y)}{\partial_y u^s (t, y)}\|_{{\cC}^{k-1}_{0}([0,T]\times \RR^+_y)}\le
e^{C\bar y}.$$

The estimate \eqref{cond-3} can be proved similarly by noting that $\partial_t u^s = \partial^2_y u^s$.

\end{proof}

\section{Well-posedness of the linearized Prandtl equation}\label{s3}
\setcounter{equation}{0}

In this section, we study the well-posedness of a linearized problem of the Prandtl
equation (\ref{eq-1}) in the Sobolev spaces by the energy method, when the background
tangential velocity is monotonic in the normal variable. Again, the main novelty
here is that unlike most of the previous works, the Crocco transformation will not be used.

In the estimates on the solutions to the linearized problem, we will see that
there is a loss of regularity  with respect to both the source term and the background
state.  And this inspires us to use the Nash-Moser-H\"ormander iteration scheme to study
the nonlinear Prandtl equation in next section.

Let $(\tilde u, \tilde v)$  be a smooth background state satisfying $$
\partial_y\tilde{u}(t, x, y)>0,\qquad \partial_x \tilde{u}+\partial_y \tilde{v}=0,
$$
and other conditions that will be specified later.
 Consider the following linearized problem of \eqref{eq-1} around  $(\tilde{u}, \tilde{v})$,
\begin{equation}\label{eq2-1}
\left\{\begin{array}{l} \partial_t u + \tilde{u} \partial_x u+\tilde{v}\partial_y u
+u \partial_x \tilde{u}+
v \partial_y \tilde{u}
-\partial^2_y u =f, \\
\partial_x u+\partial_y v=0,\\
u|_{y=0} =v|_{y=0}=0, \quad \lim\limits_{y\to+\infty} u(t,x,y) =0, \\
u|_{t\le 0} =0\,.
\end{array}\right.
\end{equation}

Unlike using the Crocco transformation, our main idea is to rewrite the problem of \eqref{eq2-1} into  a degenerate
parabolic equation with an integral term without changing the independent variables, for which we can perform the
energy estimates directly. For this purpose, we introduce the following change of unknown function:
\begin{equation*}
w(t,x,y) = \left( \frac{u}{\partial_y \tilde{u}} \right)_y(t, x, y),
\quad \mbox{that is,}\quad
u(t, x, y)=(\partial_y \tilde{u}) \int^y_0 w(t,x,\tilde{y})d \tilde{y}.
\end{equation*}

By a direct calculation, we get that for classical solutions, the  problem \eqref{eq2-1} is
equivalent to
\begin{equation}\label{eq2-2}
\left\{\begin{array}{l}
\partial_t w + \partial_x (\tilde{u} w)+\partial_y(\tilde{v} w)-2
\partial_y (\eta w)
\\
\qquad\qquad
+\partial_y(\zeta\int^y_0 w(t, x, \tilde{y})d\tilde{y})-\partial^2_y w=\partial_y\tilde{f}
,\\
\big(\partial_y w+2\eta w\big) |_{y=0} = - \tilde{f}|_{y=0}  , \\
w|_{t\le 0} =0, \end{array}\right.
\end{equation}
where
$$
\eta=\frac{\partial^2_y \tilde{u}}{\partial_y \tilde{u}},\quad \zeta
=\frac{\big(\partial_t+\tilde{u} \partial_x+
\tilde{v} \partial_y -\partial^2_y\big) \partial_y \tilde{u}}
{\partial_y \tilde{u}},\quad
\tilde{f}=\frac{f}{\partial_y \tilde{u}}.
$$

To simplify the notations in the estimates on solutions to the problem \eqref{eq2-2}, denote
\begin{equation*}
\begin{array}{lll}
\lambda_{k_1, k_2}& =& \|\tilde{u}-u^s\|_{\mathcal{B}^{k_1, k_2}_{0, 0}}
+\|\partial^{k_1}_{\mathcal{T}}\partial^{k_2}_{y}u^s
\|_{L^2_{t,y}}+\|\partial^{k_1}_{\mathcal{T}}\partial^{k_2}_{y}\tilde{v}
\|_{L^\infty_y(L^2_{t,x})}\\
&&+\|\partial_{\mathcal T}^{k_1}\partial_y^{k_2}\overline{\eta}
\|_{L^2_y(L^\infty_{t,x})}
+\|\eta-\overline{\eta}\|_{\mathcal{B}^{k_1, k_2}_{0, 0}}+
\|\zeta\|_{\mathcal{B}^{k_1, k_2}_{0, \ell}},
\end{array}\end{equation*}
and
\begin{equation*}
\lambda_{k}=\sum_{k_1+[\frac{k_2+1}{2}]\leq k}\lambda_{k_1, k_2},
\end{equation*}
where
$$
\overline{\eta}=\frac{\partial^2_y u^s}{\partial_y \tilde{u}}.
$$

The following estimate for the problem \eqref{eq2-2} can be obtained by a direct
energy method.

\begin{theo}\label{estimate-w}
For a given  positive integer $k$, suppose that the compatibility conditions of the problem
\eqref{eq2-2} hold up to order $k$.  Then for any fixed $\ell>1/2$, we have the following estimate
\begin{equation}\label{energy-A}
\|w\|_{\mathcal{A}^k_{\ell}} \le C_1(\lambda_3)
\|\tilde{f}\|_{\mathcal{A}^k_{\ell}}
+C_2(\lambda_3)\lambda_{k}\|\tilde{f}\|_{\mathcal{A}^{3}_{\ell}},
\end{equation}
where $C_1(\lambda_3), C_2(\lambda_3)$ are polynomials of $\lambda_3$ of
order less or equal to $ k$.
\end{theo}

\begin{rema}
(1) It is easy to see that the compatibility conditions for the problem
\eqref{eq2-2} up to order $k$ follow immediately from the corresponding conditions
of the problem  \eqref{eq2-1}.

(2) From the estimate \eqref{energy-A}, one can easily deduce the estimates on the
solution $(u,v)$ to the problem \eqref{eq2-1} in some weighted Sobolev spaces.
Hence, from these estimates we can obtain the well-posedness of the linearized
Prandtl equation in the Sobolev spaces.
\end{rema}

The proof of Theorem \ref{estimate-w} is based on the following lemmas.

\begin{lemm}\label{prop3.1} {\bf( $L^2$-estimate)}
Under the assumptions
of Theorem \ref{estimate-w}, for any fixed $T>0$, there is a constant $C(T)>0$ such that
\begin{equation*}
\|w\,\|^2_{\tilde{\mathcal{B}}^{0, 0}_{\lambda, \ell}} +
\lambda \|w\|^2_{\mathcal{B}^{0, 0}_{\lambda, \ell}}+
\|\partial_y w\|^2_{\mathcal{B}^{0, 0}_{\lambda, \ell}}\le
C(T)\|\tilde{f}\|^2_{\mathcal{B}^{0, 0}_{\lambda, \ell}}\,.
\end{equation*}
\end{lemm}

\begin{proof}
Multiplying \eqref{eq2-2} by $e^{-2\lambda t}\langle y\rangle^{2\ell} w$
and integrating over $\RR^2_+$, we get
\begin{align*}
&\frac12 \partial_t \|e^{-\lambda t} w(t)\|^2_{L^2_\ell(\RR^2_+)} +
\lambda \|e^{-\lambda t} w(t)\|^2_{L^2_\ell(\RR^2_+)}+
\|  e^{-\lambda t}\partial_y w(t) \|^2_{L^2_\ell(\RR^2_+)} \\
&\leq \Big(2\ell+2\|\eta\|_{L^\infty}+
\|\zeta\|_{L^\infty(\RR_x; L^2_\ell(\RR_{+, y}))}\Big)
\| e^{-\lambda t}\partial_y w(t)\|_{L^2_\ell(\RR^2_+)} \,
\| e^{-\lambda t} w(t)\|_{L^2_\ell(\RR^2_+)}\\
&+\ell\|\tilde{v}\|_{L^\infty_{-1}}\| e^{-\lambda t} w(t)
\|^2_{L^2_\ell(\RR^2_+)}+
2\ell \| e^{-\lambda t} \tilde{f}(t)\|_{L^2_{\ell-1}(\RR^2_+)}
 \,\| e^{-\lambda t} w(t)\|_{L^2_\ell(\RR^2_+)}\\
&+\| e^{-\lambda t} \tilde{f}(t)\|_{L^2_{\ell}(\RR^2_+)}
\,\| e^{-\lambda t}\partial_y w(t)\|_{L^2_\ell(\RR^2_+)},
\end{align*}
by using the boundary condition given in \eqref{eq2-2}.

Using the classical Sobolev embedding theorem, it follows
$$
4\Big(2\ell+2\|\eta\|_{L^\infty}+\|\zeta\|_{L^\infty_{t, x}
( L^2_{\ell, y})}\Big)^2+2\ell\|\tilde{v}\|_{L^\infty}\leq
(4\ell(1+\lambda_{3, 0}))^2 .
$$
Thus, by taking $\lambda\ge (4\ell(1+\lambda_{3, 0}))^2$, we get
\begin{align*}
\partial_t \|e^{-\lambda t} w(t)\|^2_{L^2_\ell(\RR^2_+)} &+
\lambda \|e^{-\lambda t} w(t)\|^2_{L^2_\ell(\RR^2_+)}
+\|  e^{-\lambda t}\partial_y w(t) \|^2_{L^2_\ell(\RR^2_+)}\\
&\leq
4 \| e^{-\lambda t}\tilde{f}(t)\|^2_{L^2_{\ell}(\RR^2_+)},
\end{align*}
which implies
\begin{equation*}
\|w\,\|^2_{\tilde{\mathcal{B}}^{0, 0}_{\lambda, \ell}} +
\lambda \|w\|^2_{\mathcal{B}^{0, 0}_{\lambda, \ell}}+
\|\partial_y w\|^2_{\mathcal{B}^{0, 0}_{\lambda, \ell}}
\le C(T)\|\tilde{f}\|^2_{\mathcal{B}^{0, 0}_{\lambda, \ell}},
\end{equation*}
for all fixed $0<T<+\infty$.
\end{proof}

\begin{lemm} {\bf (Energy estimate for tangential derivatives)}
 Under the assumptions of Theorem \ref{estimate-w}, for any fixed $T>0$,
 there is a constant $C(T)>0$ such that
\begin{align}\label{energy-k_2}
&\|w\,\|^2_{\tilde{\mathcal{B}}^{k, 0}_{\lambda, \ell}} +
\lambda \|w\|^2_{{\mathcal{B}}^{k, 0}_{\lambda, \ell}}+
\|\partial_y w\|^2_{{\mathcal{B}}^{k, 0}_{\lambda, \ell}}
\le C(T)\left( \|\tilde{f}\|^2_{{\mathcal{B}}^{k, 0}_{\lambda, \ell}}\right.\notag\\
&\qquad+\Big(\|\zeta\|^2_{{\mathcal{B}}^{k, 0}_{0, \ell}}+
\|\partial^k_{\mathcal{T}}\tilde{v}\|^2_{L^\infty_y(L^2_{t,x})}
+\|\partial_{\mathcal T}^k\overline{\eta}\|^2_{L^2_y(L^\infty_{t,x})}\Big)
\|w\|^2_{{\mathcal{B}}^{2, 1}_{\lambda, \ell}} \\
&\left.\qquad+\Big(\|\tilde{u}-u^s\|^2_{{\mathcal{B}}^{k, 0}_{0, 0}}
+\|\partial^{k}_{\mathcal{T}}u^s
\|_{L^2_{t,y}}+
\|\eta-\overline{\eta}\|^2_{{\mathcal{B}}^{k, 0}_{0, 0}}\Big)
\|w\|^2_{{\mathcal{B}}^{3, 1}_{\lambda, \ell}}\right).
\notag
\end{align}
\end{lemm}

\begin{proof}
Taking the differentiation $\partial^\beta_{\mathcal{T}}$ ($|\beta|\le k$) on
the equation in \eqref{eq2-2}, multiplying it  by $e^{-2\lambda t}\langle y\rangle^{2\ell}
 \partial^\beta_{\mathcal{T}} w$ and integrating over $\RR^2_+$, as
 in the proof of Lemma \ref{prop3.1}, for $\lambda>(4\ell(1+\lambda_{3, 0}))^2$,
 we have
\begin{equation}\label{est-1}
\begin{array}{l}
\partial_t \|e^{-\lambda t} \partial^\beta_{\mathcal{T}} w(t)
\|^2_{L^2_\ell(\RR^2_+)} +
\lambda \|e^{-\lambda t} \partial^\beta_{\mathcal{T}} w(t)
\|^2_{L^2_\ell(\RR^2_+)}+\|  e^{-\lambda t} \partial^\beta_{\mathcal{T}}
\partial_y w(t) \|^2_{L^2_\ell(\RR^2_+)} \\
\qquad \qquad \le 4\| e^{-\lambda t}\partial^\beta_{\mathcal{T}} \tilde{f}(t)
\|^2_{L^2_\ell} +A_1+A_2+A_3,
\end{array}
\end{equation}
where we have used the compatibility conditions of the problem \eqref{eq2-2} and
$$
\partial^\beta_{\mathcal{T}}\big(\tilde{v} w\big)|_{y=0}=0,
$$
by noting $\tilde{v}(t, x, y)=-\int^y_0 \partial_x \tilde{u}(t, x, \tilde{y}) d\tilde{y}$,
and $A_1, A_2$ and $A_3$ come from the commutators between $\partial^\beta_{\mathcal{T}}$
and the nonlinear terms in  \eqref{eq2-2}. For brevity,
the precise definitions  of $A_{i}$, $i=1,2,3$ are given as follows respectively.

Firstly,
\begin{align*}
A_1=&\sum_{\beta_1+\beta_2\leq \beta; |\beta_2|<|\beta|}
C^{\beta_1}_{\beta}\left|\int_{\RR^2_+}
e^{-2\lambda t}\langle y\rangle^{2\ell}(\partial^{\beta_1}_{\mathcal{T}}
\tilde{u})
(\partial^{\beta_2}_{\mathcal{T}}\partial_x w)
(\partial^{\beta}_{\mathcal{T}} w)dx dy\right|\\
&+\sum_{\beta_1+\beta_2\leq \beta; |\beta_2|<|\beta|}C^{\beta_1}_{\beta}
\left|\int_{\RR^2_+}
e^{-2\lambda t}\langle y\rangle^{2\ell}(\partial^{\beta_1}_{\mathcal{T}}
\tilde{v})
(\partial^{\beta_2}_{\mathcal{T}} w)
(\partial^{\beta}_{\mathcal{T}}\partial_y w)dx dy\right|\, .
\end{align*}
Therefore,
\begin{align*}
A_1\lesssim &\|e^{-\lambda t}\partial^{\beta}_{\mathcal{T}} w(t)
\|_{L^2_\ell(\RR^2_+)}\big\{\|\tilde{u}(t)\|_{ L^\infty(\RR^2_+)}
\|e^{-\lambda t}\partial^{|\beta|}_{\mathcal{T}} w(t)
\|_{L^2_\ell(\RR^2_+)}\\
&+\|\partial^{|\beta|}_{\mathcal{T}}( \tilde{u}(t)-u^s(t))
\|_{ L^2(\RR^2_+)}\|e^{-\lambda t}\partial_x w(t)
\|_{L^\infty_\ell(\RR^2_+)}\\
&+\|\partial^{|\beta|}_{\mathcal{T}}u^s(t)
\|_{ L^2_y}\|e^{-\lambda t}\partial_x w(t)
\|_{L^\infty_{y,\ell}(L^2_x)}\big\}\\
&+ \|e^{-\lambda t}\partial^{\beta}_{\mathcal{T}}\partial_y w(t)
\|_{L^2_\ell(\RR^2_+)}\big\{\|\tilde{v}(t)\|_{ L^\infty(\RR^2_+)}
\|e^{-\lambda t}\partial^{|\beta|-1}_{\mathcal{T}} w(t)\|_{L^2_\ell(\RR^2_+)}\\
&+ \|\partial^{|\beta|}_{\mathcal{T}}\tilde{v}(t)
\|_{L^\infty_y(L^2_x)}\|e^{-\lambda t} w(t)
\|_{L^2_{y,\ell}(L^\infty_x)} \big\},
\end{align*}
where we have used Lemma \ref{lemm2.1}.
Here and in the sequel, for simplicity, we shall use the notation $A\lesssim B$
when there exists a generic positive
constant $C$ such that $A \le C B$. Similar definition holds for $A\gtrsim B$.

Secondly, for
\begin{align*}
A_2&=\left|\int_{\RR^2_+}
e^{-2\lambda}\langle y\rangle^{2\ell} \partial^{\beta}_{\mathcal{T}}(\eta w)
(\partial^{\beta}_{\mathcal{T}} \partial_yw)dx dy\right|\,,
\end{align*}
we have
\begin{align*}
A_2\lesssim &\|\eta(t)\|_{L^\infty(\RR^2_+)}\|e^{-\lambda t}
\partial^{|\beta|}_{\mathcal{T}} w(t)\|_{L^2_\ell(\RR^2_+)}
\|e^{-\lambda t}\partial^{\beta}_{\mathcal{T}}\partial_y w(t)
\|_{L^2_\ell(\RR^2_+)}\\
&+\|\partial^{|\beta|}_{\mathcal{T}}(\eta-\overline{\eta})(t)
\|_{L^2(\RR^2_+)}\|e^{-\lambda t} w(t)\|_{L^\infty_\ell(\RR^2_+)}
\|e^{-\lambda t}\partial^{\beta}_{\mathcal{T}}\partial_y w(t)
\|_{L^2_\ell(\RR^2_+)}\\
&+\|\partial^{|{\beta}|}_{\mathcal{T}}\overline{\eta}(t)
\|_{L^\infty(\RR^2_+)}\|e^{-\lambda t} w(t)\|_{L^2_\ell(\RR^2_+)}
\|e^{-\lambda t}\partial^{\beta}_{\mathcal{T}}\partial_y w(t)
\|_{L^2_\ell(\RR^2_+)}.
\end{align*}
Finally, for
\begin{align*}
A_3&=\left|\int_{\RR^2_+}
e^{-2\lambda}\langle y\rangle^{2\ell} \partial^{\beta}_{\mathcal{T}}
\Big(\zeta \int^y_0 w d \tilde{y}\Big)
(\partial^{\beta}_{\mathcal{T}} \partial_yw)dx dy\right|\,,
\end{align*}
we get for $\ell>1/2$,
\begin{align*}
A_3\lesssim &\|\zeta(t)\|_{L^2_{y,\ell}(L^\infty_x)}
\|e^{-\lambda t}\partial^{|\beta|}_{\mathcal{T}} w(t)\|_{L^2_\ell
(\RR^2_+)}\|e^{-\lambda t}\partial^{\beta}_{\mathcal{T}}\partial_y w(t)
\|_{L^2_\ell(\RR^2_+)}\\
&+\|\partial^{|\beta|}_{\mathcal{T}}\zeta(t)\|_{L^2_\ell(\RR^2_+)}
\|e^{-\lambda t} w(t)\|_{L^2_{y,\ell}(L^\infty_x)}
\|e^{-\lambda t}\partial^{\beta}_{\mathcal{T}}\partial_y w(t)
\|_{L^2_\ell(\RR^2_+)}.
\end{align*}
Substituting these estimates of $A_1, A_2, A_3$ into \eqref{est-1}, and taking
summation over all $|\beta|\leq k$, it follows
\begin{align*}
&\partial_t \|e^{-\lambda t} \partial^k_{\mathcal{T}} w(t)
\|^2_{L^2_\ell(\RR^2_+)} +
\lambda \|e^{-\lambda t} \partial^k_{\mathcal{T}}w(t)
\|^2_{L^2_\ell(\RR^2_+)}+\|  e^{-\lambda t} \partial^k_{\mathcal{T}}
\partial_y w(t) \|^2_{L^2_\ell(\RR^2_+)}\\
&\lesssim \| e^{-\lambda t}\partial^k_{\mathcal{T}}
\tilde{f}(t)\|^2_{L^2_\ell}+
\|\partial^k_{\mathcal{T}}(\tilde{u}-u^s)\|^2_{L^2(\RR^2_+)}
\|e^{-\lambda t}\partial_x w(t)\|^2_{L^\infty_\ell(\RR^2_+)}\\
&\quad
+\|\partial^k_{\mathcal{T}}u^s\|^2_{L^2_y}
\|e^{-\lambda t}\partial_x w(t)\|^2_{L^\infty_{y,\ell}(L^2_x)}
+\|\partial^k_{\mathcal{T}}\tilde{v}\|^2_{L^\infty_y(L^2_x)}
\|e^{-\lambda t} w(t)\|^2_{L^2_{y,\ell}(L^\infty_x)}\\
&\quad+\|\partial^{k}_{\mathcal{T}}\overline{\eta}(t)
\|^2_{L^2_y(L^\infty_x)}\|e^{-\lambda t} w(t)\|^2_{L^\infty_{y,\ell}(L^2_x)}+
\|\partial^{k}_{\mathcal{T}}(\eta-\overline{\eta})(t)
\|^2_{L^2(\RR^2_+)}\|e^{-\lambda t} w(t)\|^2_{L^\infty_\ell(\RR^2_+)}\\
&\quad+\|\partial^{k}_{\mathcal{T}}\zeta(t)
\|^2_{L^2_\ell(\RR^2_+)}\|e^{-\lambda t} w(t)\|^2_{L^2_{y,\ell}(L^\infty_x)},
\end{align*}
for any $\lambda>(4\ell(1+\lambda_{3, 0}))^2$.
Integrating the above inequality on $[0, T]$, and using the compatibility condition
$$
(\partial^{k}_{t, x, y}w)|_{t=0}=0,
$$
we get for any fixed $T>0$,
\begin{align*}
&\|w\,\|^2_{\tilde{\mathcal{B}}^{k, 0}_{\lambda, \ell}} +
\lambda \|w\|^2_{{\mathcal{B}}^{k, 0}_{\lambda, \ell}}+
\|\partial_y w\|^2_{{\mathcal{B}}^{k, 0}_{\lambda, \ell}}
\le C(T)\Big\{ \|\tilde{f}\|^2_{{\mathcal{B}}^{k, 0}_{\lambda, \ell}}\\
&+\Big(\|\zeta\|^2_{{\mathcal{B}}^{k, 0}_{0, \ell}}+
\|\partial^k_{\mathcal{T}}\tilde{v}\|^2_{L^2([0, T]\times
\RR_x; L^\infty(\RR_+))}\Big)\|e^{-\lambda t} w
\|^2_{L^\infty([0, T]\times\RR_x; L^2_{y,\ell}(\RR_{+}))} \\
&+\|\tilde{u}-u^s\|^2_{{\mathcal{B}}^{k, 0}_{0, 0}}\|e^{-\lambda t}
\partial_x w\|^2_{L^\infty_\ell(\RR^2_+)}+\|\partial^k_{\mathcal{T}}
u^s\|^2_{L^2_{t,y}}\|e^{-\lambda t}
\partial_x w\|^2_{L^\infty_{y,\ell}(L^\infty([0,T],L^2_{x}(\RR)))}\\
&+
\|\partial_{\mathcal T}^k\overline{\eta}\|^2_{L^2_y(L^\infty_{t,x})}
\|e^{-\lambda t} w\|^2_{L^\infty_{y,\ell}(L^2_{t,x})}\\
&+
\|\eta-\overline{\eta}\|^2_{{\mathcal{B}}^{k, 0}_{0, 0}}
\|e^{-\lambda t} w\|^2_{L^\infty_\ell([0, T]\times\RR^2_+)}\Big\}.
\end{align*}
Using the Sobolev embedding theorem in the above inequality, the estimate
\eqref{energy-k_2} follows immediately.
\end{proof}

\begin{rema}
The estimate \eqref{energy-k_2} implies
\begin{equation}\label{energy-k_2b}
\|w\|^2_{\mathcal{B}^{k, 1}_{\lambda, \ell}}\lesssim
\|\tilde{f}\|^2_{{\mathcal{B}}^{k, 0}_{\lambda, \ell}}+
\lambda^2_{k, 0}\|w\|^2_{\mathcal{B}^{3, 1}_{\lambda, \ell}}.
\end{equation}
Moreover, using the same argument as in the above
proof together with Lemma \ref{prop3.1}, when
$\lambda>(4\ell(1+\lambda_{3, 0}))^2$, we can obtain
\begin{equation}\label{energy-3}
\|w\|_{\mathcal{B}^{k, 1}_{\lambda, \ell}}\lesssim
\|\tilde{f}\|_{{\mathcal{B}}^{k, 0}_{\lambda, \ell}},\qquad 0\leq k\le 3.
\end{equation}
\end{rema}

\smallskip
\noindent\underline{\bf Proof of Theorem \ref{estimate-w}}:
Recall from \eqref{eq2-2}  that
\begin{equation}\label{eq2-2c}
\partial^2_y w=\partial_t w+\tilde{u} \partial_x w+\tilde{v}
\partial_y w-2\partial_y(\eta w)+\partial_{y}\Big(\zeta
\int^y_0 w(t, x, \tilde{y})d\tilde{y}\Big)-\partial_{y}\tilde{f}.
\end{equation}
By applying $\partial^{k_1}_{\mathcal{T}}$ to this equation
and using Lemma \ref{lemm2.1}, we get
\begin{align*}
\|w\|_{\mathcal{B}^{k_1, 2}_{\lambda, \ell}}&\leq
\|w\|_{\mathcal{B}^{k_1+1, 0}_{ \lambda, \ell}}+
\|\tilde{u} \partial_x w\|_{\mathcal{B}^{k_1, 0}_{ \lambda, \ell}}
+\|\tilde{v} \partial_y w\|_{\mathcal{B}^{k_1, 0}_{ \lambda, \ell}}
+2\|\eta w\|_{\mathcal{B}^{k_1, 1}_{ \lambda, \ell}}\\
&+
\|(\partial_y \zeta)\int^y_0 w(t, x, \tilde{y})d\tilde{y}
\|_{\mathcal{B}^{k_1, 0}_{ \lambda, \ell}}
+\|\zeta w\|_{\mathcal{B}^{k_1, 0}_{ \lambda, \ell}}+
\|\partial_y\tilde{f}\|_{\mathcal{B}^{k_1, 0}_{ \lambda, \ell}}\\
&\lesssim (1+\|\tilde{u}\|_{L^\infty})
\|w\|_{\mathcal{B}^{k_1+1, 0}_{ \lambda, \ell}}
+ (\|\tilde{v}\|_{L^\infty}+\|\eta\|_{L^\infty})
\|w\|_{\mathcal{B}^{k_1, 1}_{ \lambda, \ell}}\\
&+\|\zeta\|_{L^\infty} \|w\|_{\mathcal{B}^{k_1, 0}_{ \lambda, \ell}}+
 \|\tilde{u}-u^s\|_{\mathcal{B}^{k_1, 0}_{0, 0}}
 \|\partial_x w\|_{L^\infty_{\lambda, \ell}}\\
&+\|\partial^{k_1}_{\mathcal{T}}u^s\|_{L^2_y(L^\infty_t)}
 \|\partial_x w\|_{L^\infty_{y, \ell}(L^2_{t,x,\lambda})}
+\|\partial^{k_1}_{\mathcal{T}}\tilde{v}\|_{L^\infty_y(L^2_{t,x})}
\|\partial_y w\|_{L^2_{y, \ell}(L^\infty_{t,x,\lambda})}  \\
&+\|\partial_{\mathcal T}^{k_1}\partial_y\overline{\eta}
\|_{L^2_y(L^\infty_{t,x})} \|w\|_{L^\infty_{y, \ell}(L^2_{t,x,\lambda})}+\|\eta-\overline{\eta}
\|_{\mathcal{B}^{k_1, 1}_{0, 0}} \|w\|_{L^\infty_{ \lambda, \ell}}\\
&+
\|\partial_y\tilde{f}\|_{\mathcal{B}^{k_1, 0}_{ \lambda, \ell}}
+\|\partial_y\zeta\|_{L^\infty_{t, x}( L^2_\ell)}
\|w\|_{\mathcal{B}^{k_1, 0}_{ \lambda, \ell}}+
\|\zeta\|_{\mathcal{B}^{k_1, 1}_{0, \ell}}\|w\|_{L^\infty_{t, x}
( L^2_{\lambda, \ell})} .
\end{align*}
Then by using the Sobolev embedding theorem, it follows that
$$
\|w\|_{\mathcal{B}^{k_1, 2}_{\lambda, \ell}}\lesssim \lambda_{2, 1}
(\|w\|_{\mathcal{B}^{k_1+1, 0}_{ \lambda, \ell}}+
\|w\|_{\mathcal{B}^{k_1, 1}_{ \lambda, \ell}})+
\|\partial_y\tilde{f}\|_{\mathcal{B}^{k_1, 0}_{ \lambda, \ell}}+
\lambda_{k_1, 1}\|w\|_{\mathcal{B}^{3, 1}_{ \lambda, \ell}}.
$$
And from \eqref{energy-k_2b} and \eqref{energy-3}, we get
$$
\|w\|_{\mathcal{B}^{k_1, 2}_{\lambda, \ell}}\lesssim \lambda_{2, 1}
\Big(\|\tilde{f}\|_{\mathcal{B}^{k_1+1, 0}_{ \lambda, \ell}}+
\|\tilde{f}\|_{\mathcal{B}^{k_1, 1}_{ \lambda, \ell}}\Big)+
(\lambda_{2, 1}\lambda_{k_1+1, 0}+\lambda_{k_1, 1})\|\tilde{f}
\|_{\mathcal{B}^{3, 0}_{ \lambda, \ell}}.
$$

For $k_2>2$,  differentiating the equation \eqref{eq2-2c}
by $\partial^{k_1}_{\mathcal{T}}\partial^{k_2-2}_y$, we have

\begin{align*}
\|w\|_{\mathcal{B}^{k_1, k_2}_{\lambda, \ell}}&\leq
\|w\|_{\mathcal{B}^{k_1+1, k_2-2}_{ \lambda, \ell}}+
\|\tilde{u} \partial_x w\|_{\mathcal{B}^{k_1, k_2-2}_{ \lambda, \ell}}
+\|\tilde{v} \partial_y w\|_{\mathcal{B}^{k_1, k_2-2}_{ \lambda, \ell}}\\
&+2\|\eta w\|_{\mathcal{B}^{k_1, k_2-1}_{ \lambda, \ell}}+
\|(\partial_y \zeta)\int^y_0 w(t, x, \tilde{y})d\tilde{y}
\|_{\mathcal{B}^{k_1, k_2-2}_{ \lambda, \ell}}\\
&+\|\zeta w\|_{\mathcal{B}^{k_1, k_2-2}_{ \lambda, \ell}}
+\|\partial_y\tilde{f}\|_{\mathcal{B}^{k_1, k_2-2}_{ \lambda, \ell}}\\
&\lesssim (1+\|\tilde{u}\|_{L^\infty})
\|w\|_{\mathcal{B}^{k_1+1, k_2-2}_{ \lambda, \ell}}
+ (\|\tilde{v}\|_{L^\infty}+\|\eta\|_{L^\infty})
\|w\|_{\mathcal{B}^{k_1, k_2-1}_{ \lambda, \ell}}\\
&+\|\zeta\|_{L^\infty} \|w\|_{\mathcal{B}^{k_1, k_2-2}_{ \lambda, \ell}}
+ \|\partial_y\zeta\|_{L^\infty_{t, x}( L^2_\ell)}
\|w\|_{\mathcal{B}^{k_1, k_2-2}_{ \lambda, \ell}} +
\|\partial_y\tilde{f}\|_{\mathcal{B}^{k_1, k_2-2}_{ \lambda, \ell}}\\
&+ \|\tilde{u}-u^s\|_{\mathcal{B}^{k_1, k_2-2}_{0, 0}}
\|\partial_x w\|_{L^\infty_{\lambda, \ell}}
+ \|\partial^{k_1}_{\mathcal{T}}\partial^{k_2-2}_y u^s\|_{L^2_y(L^\infty_{t})}
\|\partial_x w\|_{L^\infty_{y, \ell}(L^2_{t, x, \lambda})}\\
&+ \|\partial^{k_1}_{\mathcal{T}}\partial^{k_2-2}_y\tilde{v}\|_{L^\infty_y(L^2_{t, x})}
\|\partial_y w\|_{L^2_{y, \ell}(L^\infty_{t, x,\lambda})} \\
&+\|\partial_{\mathcal T}^{k_1}\partial_y^{k_2-1}\overline{\eta}\|_{L^2_y(L^\infty_{t,x})}
 \|w\|_{L^\infty_{y, \ell}(L^2_{t,x,\lambda})}
+\|\eta-\overline{\eta}\|_{\mathcal{B}^{k_1, k_2-1}_{0, 0}}
 \|w\|_{L^\infty_{ \lambda, \ell}}\\
&
+\|\zeta\|_{\mathcal{B}^{k_1, k_2-1}_{0, \ell}}\|w
\|_{L^\infty_{y, \ell}(L^2_{t,x,\lambda})} \, .
\end{align*}
Therefore, we get
\begin{align*}
\|w\|_{\mathcal{B}^{k_1, k_2}_{\lambda, \ell}}&\lesssim \lambda_{2,1}
( \|w\|_{\mathcal{B}^{k_1+1, k_2-2}_{ \lambda, \ell}}+
 \|w\|_{\mathcal{B}^{k_1, k_2-1}_{ \lambda, \ell}})+
 \lambda_{k_1,k_2-1}\|w\|_{\mathcal{B}^{3, 1}_{ \lambda, \ell}}
+\|\tilde{f}\|_{\mathcal{B}^{k_1, k_2-1}_{ \lambda, \ell}},
\end{align*}
which immediately implies
\begin{align*}
\|w\|_{\mathcal{B}^{k_1, k_2}_{0, \ell}}&\lesssim \lambda_{2,1}
( \|w\|_{\mathcal{B}^{k_1+1, k_2-2}_{ 0, \ell}}+
 \|w\|_{\mathcal{B}^{k_1, k_2-1}_{ 0, \ell}})+
 \lambda_{k_1,k_2-1}\|w\|_{\mathcal{B}^{3, 1}_{ 0, \ell}}
+\|\tilde{f}\|_{\mathcal{B}^{k_1, k_2-1}_{0, \ell}},
\end{align*}
by fixing  $\lambda>(4\ell(1+\lambda_{3, 0}))^2$.

The proof of Theorem \ref{estimate-w} can then be completed by induction on $k_2$.

\section{Iteration scheme for the nonlinear Prandtl equation}\label{section4}
\setcounter{equation}{0}

{}From the estimate \eqref{energy-A} given in Theorem \ref{estimate-w}, we see that there
is a loss of regularity in the  solutions to the linearized Prandtl equation with respect
to the source term and the background state. In order to take care of this loss,
 we are going to apply the Nash-Moser-H\"ormander iteration scheme, cf.
 \cite{ali-gerard, hami, hor, moser, nash}, to study the nonlinear problem \eqref{eq-1}.

\subsection{The smoothing operators}

For a function $f$ defined on $\Omega=[0, +\infty[\times\RR_x\times\RR^+_y$,
let $\tilde{f}$ be its extension to $\RR^3$ by $0$. Then for a large constant
$\theta$, introduce a family of smoothing operators $S_\theta$:
$$
(S_{\theta} f) (t, x, y)=\int \rho_{_\theta}(\tau)\rho_{_\theta}(\xi)
\rho_{_\theta}(\eta)\tilde{f}(t-\tau {+}\theta^{-1}, x-\xi, y-\eta+\theta^{-1})d\tau d\xi
d \eta,
$$
where $\rho_{_\theta}(\tau)= \theta\,\rho(\theta\tau),
\,\rho\in C^\infty_0(\RR)$ with Supp $\rho\subseteq [-1, 1]$ and $\|\rho\|_{L^1}=1$. One has
\begin{equation*}
\{S_\theta\}_{\theta>0}:~ {\mathcal{A}}^0_{ \ell}
(\Omega)\longrightarrow \cap_{s\ge 0}{\mathcal{A}}^s_{
\ell}(\Omega),
\end{equation*}
together with
\begin{equation}\label{f2.5}
\begin{cases}
 \|S_\theta u\|_{{\mathcal{A}}^s_{
\ell}}\le
C_\rho\theta^{(s-\alpha)_+}\|u\|_{{\mathcal{A}}^\alpha_{ \ell}},
\quad {\rm for ~all}~s,
\alpha\ge 0,\\[2mm] \|(1-S_\theta) u\|_{{\mathcal{A}}^s_{
\ell}}\le C_\rho\theta^{s-\alpha}\|u\|_{{\mathcal{A}}^\alpha_{\ell}}, \quad {\rm for ~all}~0\le
s\le\alpha,
\end{cases}
\end{equation}
where the constant $C_\rho$ depends only on the function $\rho$ and the orders of differentiations
$s$ and $\alpha$. For the smoothing parameter, we set $\theta_n=\sqrt{\theta_0^2+n}$ for any $n\ge 1$ and
a large fixed constant $\theta_0$. We have also
\begin{equation}\label{f2.5-3}
\|(S_{\theta_n}-S_{\theta_{n-1}}) u\|_{{\mathcal{A}}^s_{\ell}}\le
C_\rho\theta_n^{s-\alpha}\Delta\theta_n \|u\|_{{\mathcal{A}}^\alpha_{\ell}}, \quad {\rm for ~all}~s,
\alpha\ge 0,
\end{equation}
where $\Delta\theta_n=\theta_{n+1}-\theta_n$.

The operator $S_\theta$ acting on the other three spaces introduced in  Section  \ref{s2.1} shares the same properties.

The following commutator estimates will be used frequently later,
\begin{lemm}\label{lemm-commutators}
For any proper function $f$, we have
\begin{equation}\label{commutator-1}
\|[\frac{1}{\partial_y u^s}, S_{\theta}](\partial_yf)
\|_{\cA^k_\ell}\le C_k\|\frac{f}{\partial_y u^s}\|_{\cA^k_\ell},\end{equation}
and \begin{equation}\label{commutator-2}
\|\partial_y[\frac{1}{\partial_y u^s}, \partial_y S_{\theta}]
f\|_{\cA^k_\ell}\le C_k\theta \|\frac{f}{\partial_y u^s}\|_{\cA^k_\ell},
\end{equation}
with the constant $C_k$ depends on the constant in \eqref{result-u-s}.
Similar inequalities hold for the norms  $\|\,\cdot\,\|_{{\mathcal B}_{\ell}^{k_1, k_2}},
\|\,\cdot\,\|_{{\mathcal C}_{\ell}^{k}}$ and $\|\,\cdot\,\|_{{\mathcal D}_{\ell}^{k}} $.
\end{lemm}
\begin{proof}
This lemma can be proved in a classical way, cf. \cite{Piriou, hormander}.
 To be self-contained,
we give a brief proof of the estimate \eqref{commutator-1} here. Note that
 \eqref{commutator-2} can be proved similarly.

From the definition, we have
\begin{equation}\label{inden}
\begin{array}{ll}
&[\frac{1}{\partial_y u^s}, S_{\theta}]b=\frac{S_{\theta} (b)}{\partial_y u^s}(t, x, y)-S_{\theta}\left(\frac{b}{\partial_y u^s}\right)(t, x, y)\\
&=
\int\rho_{_\theta}(\tau)\rho_{_\theta}(\xi)\rho_{_\theta}(\eta)
\left(\frac{\partial_y u^s(t-\tau+\theta^{-1} , y-\eta+\theta^{-1})-\partial_y u^s(t, y)}{\partial_y u^s(y)}
\right)\\
&\qquad\qquad\qquad\times
\left(\frac{ \tilde{b}}{\partial_y u^s}\right)(t-\tau+\theta^{-1}, x-\xi, y-\eta+\theta^{-1}) d \tau d\xi d \eta
\\
&=\theta^{-1}
\int \rho_{_\theta}(\xi)\Big( \tilde{\rho}_{_\theta}(\tau)\rho_{_\theta}(\eta)a_1(t, \tau, y, \eta,\theta)+
\rho_{_\theta}(\tau)\tilde{\rho}_{_\theta}(\eta)a_2
(t, \tau, y, \eta,\theta) \Big)
\\
&\qquad\qquad\qquad\times\left(\frac{ \tilde{b}}{\partial_y u^s}\right)(t-\tau+\theta^{-1},
x-\xi, y-\eta+\theta^{-1}) d \tau d\xi d \eta,
\end{array}\end{equation}
with $\tilde{\rho}_{_\theta}(\tau)= \theta\,\rho(\theta\,\tau)(1-\theta\,\tau),
\tilde{\rho}_{_\theta}(\eta)= \theta\,\rho(\theta\,\eta)(1-\theta\,\eta) $, and
\begin{align*}
&a_1(t, \tau, y, \eta,\theta)
=\int^1_0\frac{{\partial_t\partial_y u^s(t+\lambda(\theta^{-1}-\tau),  y+\lambda(\theta^{-1}-\eta))}}{\partial_y u^s(t, y)}
d \lambda,
\end{align*}
and
\begin{align*}
&a_2(t, \tau, y, \eta,\theta)
=\int^1_0\frac{{\partial^2_y u^s(t+\lambda(\theta^{-1}-\tau),  y+\lambda(\theta^{-1}-\eta))}}{\partial_y u^s(t, y)}
d \lambda.
\end{align*}
Using \eqref{cond-2} and \eqref{cond-3}, for $j=1, 2$, we have
\begin{equation}\label{Comm}
\sup _{0\leq |\tau|, |\eta|\leq \theta^{-1}\leq R_0}\|a_j(\,\cdot\, , \tau,\,\cdot\, , \eta,\, \theta)\|_{{\cC}^{k+1}_{0}}\le \tilde{C}(T).
\end{equation}
Thus, from \eqref{inden} we get the estimate \eqref{commutator-1} immediately.
\end{proof}

\subsection{The iteration scheme}\label{s2.3}
Denote by
\begin{equation*}
\mathcal{P}(u, v)
=\partial_tu + u\partial_xu+v\partial_y u- \partial^2_yu,
\end{equation*}
the nonlinear operator associated with problem \eqref{eq-1}, and its linearized operator around
$(\tilde u, \tilde v)$ by
\begin{equation*}
\mathcal{P}'_{(\tilde u, \tilde v)}(u, v)=
 \partial_t u+ \tilde u\partial_x u + \tilde v \partial_y u+
 u\partial_x \tilde u
+ v \partial_y \tilde u- \partial^2_y u.
\end{equation*}

In this subsection, we introduce an iteration
scheme in order to construct
an approximate solution sequence $\{(u^n, v^n)\}$ to the
problem \eqref{eq-1}.

For a fixed integer $k\ge 0$, suppose that the
initial data in the Prandtl equation \eqref{eq-1}
satisfies the compatibility conditions up to order $k$,
and that $(u, v)$ is a classical solution.
If we set $\tilde{u}=u-u^s$ with $u^s(t,y)$ being the heat profile defined in
Section \ref{s2.2-1}, then it is easy to see that
\begin{equation}\label{tilde-u}
\begin{cases}
\partial_t \tilde{u}+(\tilde{u}+u^s)\tilde{u}_x+
v \partial_y(\tilde{u}+u^s)-\tilde{u}_{yy}=0, \quad (x, y)\in\RR^2_+,\,\,t> 0,\\
\partial_x\tilde{u}+\partial_yv=0, \\
\tilde{u}|_{y=0}=v|_{y=0}=0, \quad \lim\limits_{y\to +\infty}\tilde{u}=0,
\\
\tilde{u}|_{t=0}=\tilde{u}_0(x,y) .\end{cases}.
\end{equation}
The compatibility conditions for \eqref{tilde-u} follow from those for \eqref{eq-1} immediately.

\smallskip
\noindent
\underline{\it The zero-th order approximate solution}: Denote
$$
\tilde{u}^j_0(x,y)=\partial_t^j\tilde{u}|_{t=0}, \qquad v^j_0(x,y)=\partial_t^jv|_{t=0}.
$$
Then from the compatibility conditions for \eqref{tilde-u},  $\{\tilde{u}^j_0, v_0^j\}_{j\le k}$
are defined directly by $\tilde{u}_0(x,y)$. We are going to construct the zero-th order
approximate solution $(\tilde{u}^0, v^0)$ of \eqref{tilde-u}, such that
$$
\partial_t^j\tilde{u}^0|_{t=0}=\tilde{u}^j_0(x,y), \qquad \partial_t^j
v^0|_{t=0}=v^j_0(x,y),\quad 0\le j\le k,
$$
and $(u^0, v^0)=
(u^s+\tilde{u}^0,  v^0)$ satisfying
\begin{equation}\label{eq-first}
\begin{cases}
\partial_xu^0+\partial_yv^0=0, \quad (x, y)\in\RR^2_+,\,\,t\ge 0, \\
u^0|_{y=0}=v^0|_{y=0}=0, \lim\limits_{y\to+\infty}u^0=1,\\
u^0|_{t=0}=u_0(x,y)\,.\end{cases}.
\end{equation}
Other properties of $(u^0, v^0)$ will be studied in more details in  Section \ref{s5.1}.

\smallskip
\noindent
\underline{\it The Nash-Moser iteration scheme}:
Assume that for all $k=0, \ldots, n$, we have constructed the approximate solutions
$(u^{k}, v^{k})$ of \eqref{eq-1} satisfying the same conditions  given in
\eqref{eq-first} for $(u^0, v^0)$. We now construct the $(n+1)$-th approximation
solution $(u^{n+1}, v^{n+1})$ as follows. Set
\begin{equation}\label{u-n+1}
u^{n+1}=u^{n}+\delta u^{n}=u^s+\tilde{u}^{n}+\delta u^{n},
\qquad v^{n+1}=v^{n}+\delta v^{n}\,,
\end{equation}
where the increment $(\delta u^{n}, \delta v^{n})$ is the solution of the following initial-boundary value problem,
\begin{equation}\label{prob-u-n}
\begin{cases}
\mathcal{P}'_{(u^{n}_{\theta_n},
v^{n}_{\theta_n})}(\delta u^{n}, \delta v^n)
=f^n, \cr
\partial_x (\delta u^n)+\partial_y (\delta v^n)=0, \cr
\delta u^n|_{y=0}=\delta v^n|_{y=0}=0, \qquad \lim\limits_{y\to+\infty}\delta u^n=0, \cr
\delta u^{n}|_{t\le 0}=0,
\end{cases}
\end{equation}
where  $u^n_{\theta_n}=u^s+S_{\theta_n} \tilde{u}^n$ and $v^n_{\theta_n}=S_{\theta_n}v^n$.

Now, we define the source term $f^n$ for the problem  \eqref{prob-u-n} in order to have
the convergence of the approximate solution sequence $(u^n, v^n)$ to the solution of
the Prandtl equation \eqref{eq-1} as $n$ goes to infinity.
Obviously, we have the following identity,
\begin{equation}\label{identity}
\mathcal{P}(u^{n+1}, v^{n+1})-\mathcal{P}(u^{n}, v^{n})
=\mathcal{P}'_{(u^{n}_{\theta_n}, v^{n}_{\theta_n})}(\delta u^{n},
\delta v^{n})+e_{n}\,,
\end{equation}
where
$$
e_{n}=e_{n}^{(1)}+e_{n}^{(2)}.
$$
Here
$$
\begin{array}{ll}
e^{(1)}_n&=\mathcal{P}(u^{n}+\delta u^n, v^{n}+\delta
v^n)-\mathcal{P}(u^n, v^n)
-\mathcal{P}'_{(u^n, v^n)}(\delta u^n, \delta v^n)\\
&=\delta u^n \partial_x(\delta u^n)+\delta v^n \partial_y (\delta u^n),
\end{array}
$$
is the error from the Newton iteration scheme, and
$$
\begin{array}{ll}
e_{n}^{(2)}& =\mathcal{P}'_{(u^n, v^n)}(\delta u^n, \delta v^n)
-\mathcal{P}'_{(u^n_{\theta_n}, v^n_{\theta_n})}(\delta u^n, \delta
v^n)\\
&=\big((1-S_{\theta_n})
(u^n-u^s)\big) \partial_x(\delta u^n)+
\delta u^n \partial_x \big((1-S_{\theta_n})(u^n-u^s)\big)\\[2mm]
&\qquad+ \delta v^n \partial_y \big((1-S_{\theta_n}) (u^n-u^s)\big)
+\big((1-S_{\theta_n}) v^n\big) \partial_y(\delta u^n),
\end{array}
$$
is the error coming from mollifying the coefficients.

{}From \eqref{identity}, we have
\begin{equation}\label{identity-1}
\mathcal{P}(u^{n+1}, v^{n+1})=\sum_{j=0}^n( \mathcal{P}'_{
(S_{\theta_j}u^{j}, S_{\theta_j}v^{j})}(\delta u^{j}, \delta v^j)+{e}_{j})+{f}^a,
\end{equation}
with
$$
{f}^a=\mathcal{P}(u^0, v^0):=\partial_t u^0 + u^0
\partial_x u^0 + v^0\partial_y u^0- \partial^2_y u^0.
$$

Note that if the approximate solution $(u^{n}, v^{n})$ converges to a solution of the problem (\ref{eq-1}), then the right hand
side in the equation (\ref{identity-1}) should go to zero when $n\to +\infty$. Thus, it is
natural to require that $(\delta u^n, \delta v^n)$ satisfies
the following equation for all $n\ge 0$,
$$
\mathcal{P}'_{(u^{n}_{\theta_n},
v^{n}_{\theta_n})}(\delta u^{n}, \delta v^n)
=f^n,
$$
where $f^n$ is defined by
$$
\sum_{j=0}^nf^j=-S_{\theta_n}(\sum_{j=0}^{n-1}e_{
j})-S_{\theta_n}{f}^a,
$$
by induction on $n$. Obviously, we have
\begin{equation}\label{fn-1}
\begin{cases}
f^0=-S_{\theta_0}{f}^a, \qquad f^1=
(S_{\theta_0}-S_{\theta_1}) {f}^a+S_{\theta_1} e_0,\\[2mm]
f^n=(S_{\theta_{n-1}}-S_{\theta_n})\big(\sum\limits^{n-2}_{j=0}
e_j\big)-S_{\theta_n} e_{n-1}+
(S_{\theta_{n-1}}-S_{\theta_n}){f}^a, \quad \forall n\ge 2.
\end{cases}
\end{equation}


\section{Existence of the classical solutions} 
\label{section5}
\setcounter{equation}{0}

In this section, we study the iteration scheme \eqref{u-n+1}-\eqref{prob-u-n} with $f^n$ being given in \eqref{fn-1},
by using the estimate  \eqref{energy-A} given in Theorem \ref{estimate-w}. To do this, let us first state the main
assumption (MA) on the initial data $\tilde{u}_0(x,y)$ of \eqref{tilde-u} as follows:
\begin{quote}
{\bf (MA)} For any fixed integers $\tilde{k}\ge 7$, $k_0\ge \tilde{k}+2$,  and a real number $\ell>\frac{1}{2}$, suppose that $\tilde{u}_0\in \cA^{2k_0+1}(\RR_+^2)$ satisfies the compatibility conditions for the problem \eqref{tilde-u} up to order $k_0$, and
$$\|\tilde{u}_0\|_{\cA^{2k_0+1}_{\ell}(\RR_+^2)}+
\|\frac{\partial_y\tilde{u}_0}{\partial_yu^s_0}\|
_{\cA^{2k_0+1}_{\ell}(\RR_+^2)}\le \epsilon,$$
for a small quantity $\epsilon>0$ depending on the norms of $u_0^s(y)$.
\end{quote}


\subsection{The zero-th order approximation} \label{s5.1}

Let us construct the zero-th order approximate solution $(u^0, v^0)$ satisfying
\eqref{eq-first} to the problem \eqref{eq-1}. As  mentioned in  Section  \ref{s2.3}, from
the equation \eqref{tilde-u} one can easily obtain
$\tilde{u}^j_0(x,y)=\partial_t^j\tilde{u}(0, x,y)$ and  $v^j_0(x,y)=\partial_t^jv(0,x,y)
$
in terms of $\tilde{u}_0(x,y)$ for all $0\leq j\leq k_0$, and then
have the following relations
\begin{equation}\label{compatibility-b}
\left\{\begin{array}{l}
u^j_0(x,y)=\partial^2_yu^{j-1}_0-\sum\limits_{k=0}^{j-1}
C^k_{j-1}
\Big(u^k_0\partial_xu^{j-1-k}_0+v^k_0\partial_y u^{j-1-k}_0\Big),\\
v^j_0(x,y)=-\int_0^y\partial_xu^j_0(x,\xi)d\xi,
\end{array}\right.
\end{equation}
by induction on $j$, with $u^j_0(x,y)=\tilde{u}^j_0(x,y)+(\partial_t^j u^s)(0,y)$.

To construct $(\tilde{u}^0, v^0)$ satisfying
$$
\partial_t^j\tilde{u}^0|_{t=0}=\tilde{u}^j_0(x,y), \qquad \partial_t^j
v^0|_{t=0}=v^j_0(x,y),\quad 0\le j\le k_0,
$$
we can simply define
\begin{equation}\label{u-a-1}
\tilde{u}^0(t,x,y)=\sum_{j=0}^{k_0}\frac{t^j}{j!}\tilde{u}^j_0(x,y),
\quad
v^0(t,x,y)=\sum_{j=0}^{k_0}\frac{t^j}{j!}v^j_0(x,y)\, .
\end{equation}

For this approximate solution,  we have

\begin{lemm}\label{u-a} Under the assumption (MA),
 for any fixed $T>0$, there is a constant $C=C(k_0, T)$ depends only on $k_0$ and $T$ such that,
\begin{equation}\label{u-a-2}
\|\tilde{u}^0\|_{\tilde{{\cA}}^{k_0+1}_{\ell}([0,T]\times \RR_+^2)}\le C\epsilon,
\quad \|v^0\|_{\cD^{k_0}([0,T]\times \RR_+^2)}\le C\epsilon,
\end{equation}
and
\begin{equation}\label{u-a-2-1}
\|{f}^a\|_{\tilde{\cA}^{k_0}_{\ell}([0,T]\times \RR_+^2)}\le C\epsilon.
\end{equation}
Here, we have used the notations
$$\|\tilde{u}^0\|_{\tilde{{\cA}}^{k}_{\ell}([0,T]\times \RR_+^2)}:=
\sum_{j=0}^{k}\|\tilde{u}^0\|_{W^{j,\infty}(0,T;\cA^{k-j}
_{\ell}(\RR_+^2))},$$
and
$$
{f}^a=\partial_t \tilde{u}^0+(\tilde{u}^0+u^s)\partial_x\tilde{u}^0+
v^0 \partial_y(\tilde{u}^0+u^s)-\partial_y^2\tilde{u}^0.
$$
\end{lemm}

\begin{proof}
Since $\tilde{u}_0\in \cA_{\ell}^{2k_0+1}(\RR_+^2)$, it follows immediately that
$$
v_0(x,y)=-\int_0^y\partial_x\tilde{u}_0(x,\eta)d\eta\in \cD^{2k_0}(\RR_+^2),
$$
which implies
$$
\tilde{u}_0^1=-(\tilde{u}_0+u_0^s)\partial_x\tilde{u}_0-
v_0 \partial_y(\tilde{u}_0+u_0^s)+\partial_y^2\tilde{u}_{0}
\in \cA_{\ell}^{2k_0}(\RR_+^2),
$$
and
$$
v_0^1(x,y)=-\int_0^y\partial_x\tilde{u}^1_0(x,\eta)d\eta\in \cD^{2k_0-1}(\RR_+^2).
$$

In this way, using \eqref{compatibility-b} and by induction on $j$ we can deduce
$$
\tilde{u}_0^j\in \cA_{\ell}^{2k_0+1-j}(\RR_+^2), \qquad v_0^j\in \cD^{2k_0-j}(\RR_+^2),
$$
for all $j\le k_0$, and
$$
\|\tilde{u}_0^j\|_{\cA_{\ell}^{2k_0-j+1}(\RR_+^2)}\le C(j)\epsilon,
$$
for a constant $C(j)$ depending only on $j$. Then,  \eqref{u-a-2} and \eqref{u-a-2-1}
follow immediately from the construction \eqref{u-a-1}.
\end{proof}

\begin{rema}
Denoting $u^0=\tilde{u}^0+u^s$, it is easy to see that
$(u^0, v^0)$ is an approximate solution to the original problem \eqref{eq-1} satisfying
\begin{equation}\label{f-a}
\begin{cases}
u_t^0+u^0u^0_x+v^0 u^0_y-u^0_{yy}=f^a,
\quad t>0, \quad (x, y)\in\RR^2_+,\\
\partial_xu^0+\partial_yv^0=0, \quad (x, y)\in\RR^2_+,\,\,t\ge 0, \\
(u^0, v^0)|_{y=0}=0, \quad \lim\limits_{y\to +\infty} u^0=1,\\
u^0|_{t=0}=u_0(x,y),\end{cases}
\end{equation}
with $\partial_t^j{f}^a|_{t=0}=0$ for all $0\le j\le k_0-1$.
\end{rema}

\begin{lemm}
Under the assumptions (MA), for any fixed~ $T>0$, there is a constant $C(k, T)$ such that,
\begin{equation}\label{f-a-1}
\|\frac{\partial_y\tilde{u}^0}{\partial_yu^s}\|_{\cA^{k_0+1}_{\ell}([0,T]\times \RR_+^2)}
+\|\frac{{f}^a}{\partial_yu^s}\|_{\cA^{k_0}_{\ell}([0,T]\times \RR_+^2)}\le C\epsilon.
\end{equation}
\end{lemm}

This result can be proved in a way similar to the proof of Lemma \ref{u-a}, so we omit it for brevity.

\begin{rema}
From the smallness  of the first term given in \eqref{f-a-1}, it is easy to see that for a fixed $T>0$,
there is a small $\epsilon$ such that when the conditions of Lemmas \ref{u-a} and \ref{f-a} hold,
we have the monotonicity
$$
\partial_y u^0(t, x, y)>0,\qquad (t, x, y)\in [0, T]\times\RR^2_+.
$$
\end{rema}

\subsection{Estimates of the approximate solutions}

Obviously, the problem \eqref{prob-u-n} can be written as
 \begin{equation}\label{iteration-n}
\left\{\begin{array}{l}
\partial_t (\delta u^n)+ u^n_{\theta_n}\partial_x
(\delta u^n) +v^n_{\theta_n}
\partial_y (\delta u^n)+\delta u^n\partial_x (u^n_{\theta_n})\\
\quad+ \delta v^n \partial_y (u^n_{\theta_n})- \partial^2_y (\delta u^n)=f^n,\\
\partial_x(\delta u^n) + \partial_y(\delta v^n) =0,\\
\delta u^n|_{y=0}=\delta v^n|_{y=0} =0,
\quad \lim\limits_{y\to+\infty}\delta u^n=0,\\
\delta u^n|_{t=0}=0\, ,
\end{array}\right.
\end{equation}
where
$$
u^n_{\theta_n}=u^s+S_{\theta_n}(\tilde{u}^0+\sum_{0\leq j\leq n-1}\delta u^j), \quad
v^n_{\theta_n}=S_{\theta_n}(v^0+\sum_{0\leq j\leq n-1}\delta v^j).
$$

Set
$$
w^n=\partial_y\left(\frac{\delta u^n}{\partial_y u^n_{\theta_n}}\right).
$$
As in  Section \ref{s3}, from \eqref{iteration-n}, we know that $w^n$ satisfies
\begin{equation}\label{w-n}
\left\{\begin{array}{l}
\partial_t w^n+ \partial_x(u^n_{\theta_n} w^n) +
\partial_y(v^n_{\theta_n} w^n)-2\partial_y(\eta^n w^n)
\\ \qquad
+ \partial_y \left(\zeta^n\int^y_0w^n(t, x, \tilde{y})
d\tilde{y}\right)- \partial^2_y w^n=\partial_y \tilde{f}^n,\\
\big(\partial_yw^n+2\eta^n w^n\big)\big|_{y=0}= -\tilde{f}^n\big|_{y=0} ,\\
w^n|_{t=0}=0\, ,
\end{array}\right.
\end{equation}
where
$$
\eta^n=\frac{\partial^2_y u^n_{\theta_n}}{\partial_y u^n_{\theta_n}},\quad \quad
\zeta^n=\frac{(\partial_t+u^n_{\theta_n} \partial_x+
 v^n_{\theta_n}\partial_y -\partial^2_y)
 \partial_y u^n_{\theta_n}}{\partial_y u^n_{\theta_n}},
$$
and
\begin{equation}\label{tilde-f-n}
\tilde{f}^n=\frac{f^n}{\partial_y u^n_{\theta_n}}=
 \frac{(S_{\theta_{n-1}}-S_{\theta_n})\big(\sum^{n-2}_{j=0} e_j
 \big)-S_{\theta_n} e_{n-1}+
(S_{\theta_{n-1}}-S_{\theta_n}){f}^a}{\partial_y u^n_{\theta_n}}.
\end{equation}

From the above main assumption (MA) and the construction of the  approximate
solution $(\tilde{u}^0, v^0)$ to the problem \eqref{tilde-u},  it is easy to
show by induction on $n$ that the compatibility conditions for the problem
\eqref{w-n} up to order $k_0$ hold for all $n\ge 0$.

Similar to Section \ref{s3},  set
\begin{eqnarray*}
\lambda^n_{k_1,k_2}&=&\|u^n_{\theta_n} -u^s\|_{\mathcal{B}_{0,0}^{k_1,k_2}}
+\|\partial^{k_1}_{\mathcal T}\partial^{k_2}_yu^s\|_{L^2_y(L^\infty_t)}
+\|\partial^{k_1}_{\mathcal T}\partial^{k_2}_yv^n_{\theta_n}\|_{L^\infty_y(L^2_{t,x})}\\ & &
+\|\partial^{k_1}_{\mathcal T}\partial^{k_2}_y\bar\eta^n
\|_{L^2_y(L^\infty_{t,x})}  +\|\eta^n-\bar\eta^n
\|_{\mathcal{B}_{0,0}^{k_1,{k_2}}} +\|\zeta^n
\|_{\mathcal{B}_{0,l}^{k_1,{k_2}}},\end{eqnarray*}
with $\bar{\eta}^n=\frac{\partial_y^2u^s}{\partial_yu_{\theta_n}^n}$, and
\begin{eqnarray*}
\lambda^n_k&=&\sum_{k_1+[\frac{k_2+1}{2}]\le k} \lambda^n_{k_1,k_2}.
\end{eqnarray*}
That is,
\begin{eqnarray}\label{lambda-n}
\lambda^n_k & = & \|u^n_{\theta_n} -u^s\|_{\mathcal{A}_{0}^{k}}
+\|u^s\|_{\mathcal{C}_{0}^{k}}
+\|v^n_{\theta_n}\|_{\mathcal{D}_{0}^{k}}\\[2mm]
&&+\|\bar\eta^n
\|_{\mathcal{C}_{0}^{k}} +\|\eta^n-\bar\eta^n
\|_{\mathcal{A}_{0}^{k}} +\|\zeta^n
\|_{\mathcal{A}_{l}^{k}}\notag.
\end{eqnarray}

By applying Theorem \ref{estimate-w} to the problem \eqref{w-n}, we have

\begin{prop} \label{prop-w-n}
Under the  main assumption (MA), the solution $w^n$ to the
problem \eqref{w-n} satisfies
\begin{equation}\label{energy-A-1}
\|w^n\|_{\mathcal{A}^k_{ \ell}} \le C_1(\lambda^n_3)
\|\tilde{f}^n\|_{\mathcal{A}^k_{\ell}}
+C_2(\lambda^n_3)\lambda^n_k\|\tilde{f}^n\|_{\mathcal{A}^{3}_{\ell}},
\end{equation}
where $C_1(\lambda^n_3), C_2(\lambda^n_3)$ are polynomials of
$\lambda^n_3$ of order  less or equal to $k$.
\end{prop}

The key step in proving the convergence of the Nash-Moser-H\"ormander iteration scheme
\eqref{u-n+1}-\eqref{prob-u-n} is given by the following result.

\begin{theo}\label{theo-wn}
Under the main assumption (MA),
there exists a positive constant $C_0$, such that
\begin{align}\label{wn}
\|w^n\|_{\mathcal{A}_{\ell}^{k}}\le C_0 \,\epsilon \,
\theta^{\max\{3-\tilde{k},\, k-\tilde{k}\}}_n \Delta \theta_n,
\end{align}
holds for all $n\ge 0$, $0\le k\le k_0$ where $\theta_n=\sqrt{\theta_0^2+n}$
and $\Delta \theta_n=\theta_{n+1}-\theta_{n}$.
\end{theo}

Theorem \ref{theo-wn} will be proved by induction on $n$. First of all,
to apply Proposition \ref{prop-w-n}, we need to estimate
$\lambda_k^n$ and $\tilde{f}^n$
by induction on $n$ also. For this purpose, we first give the following estimates,
some of the proofs  being postponed to  Section  \ref{section7}. The proof of Theorem \ref{theo-wn}
will be completed at the end of this subsection.

\begin{lemm}\label{lemm-delta-u-j}
Suppose that the main assumption (MA), and
\eqref{wn} for $w^j$, $0\le j\le n-1$, hold.  Then there is a constant $C_1>0$, such that
\begin{align}
\|\delta u^j\|_{\mathcal{A}_{\ell}^{k}}\leq
C_1\epsilon\, \theta_j^{\max\{3-\tilde{k},\,k-\tilde{k}\}}\Delta \theta_j, \qquad 0\le k\le k_0,
\label{esti-delta-u}
\end{align}
\begin{equation}\label{Linfty}
\|\frac{\delta u^j}{\partial_y u^s}\|_{L^\infty([0,T]\times \RR_+^2)}\le C_1
\epsilon\theta_j^{3-\tilde{k}}\Delta\theta_j,
\end{equation}
and
\begin{equation}\label{D-k-1}
\|\frac{\delta u^j}{\partial_y u^s}
\|_{{\cD}^{k-1}_{0}}\le C_1\epsilon
\,\theta_j^{\max\{3-\tilde{k},\,k-1-\tilde{k}\}}
\Delta\theta_j, \qquad 1\le k\le k_0,
\end{equation}
hold for all $0\le j\le n-1$.
\end{lemm}

Since the proof of this lemma is technical, it will be given in  Section  \ref{section7}.

As $\delta v^j(t,x,y)=-\int_0^y(\partial_x\delta u^j)(t,x,\tilde{y})d\tilde{y}$,
from \eqref{esti-delta-u} we immediately have

\begin{lemm}\label{lemm-delta-v-j}
Under the same assumptions as for Lemma \ref{lemm-delta-u-j}, there is a constant $C_2>0$, such that
\begin{equation}\label{esti-delta-v}
\|\delta v^j\|_{\cD_0^k}\le C_2 \epsilon\theta_j^{\max\{3-\tilde{k},k+1-\tilde{k}\}}\Delta\theta_j,
\qquad 0\le k\le k_0-1,
\end{equation}
holds for all $0\le j\le n-1$.
\end{lemm}

Based on Lemma \ref{lemm-delta-u-j} and Lemma \ref{lemm-delta-v-j}, we have

\begin{lemm}\label{lemm-u-v-n}
Under the same assumptions as for Lemma \ref{lemm-delta-u-j}, there is a constant $C_3>0$, such that
\begin{equation}\label{esti-u-n}
\|u^n-u^s\|_{\mathcal{A}_{l}^{k}}\leq C_3\epsilon
\theta_n^{\max\{0,\,k+1-\tilde{k}\}}, \qquad 0\le k\le k_0,
\end{equation}
\begin{equation}\label{L-infinity}
\|v^n_{\theta_n}\|_{L^\infty([0,T]\times \RR_+^2)}+\|v^n_{\theta_n}
\|_{L^\infty([0,\infty)_y, L^2([0,T]\times \RR_x)}\leq C_3,
\end{equation}
and
\begin{equation}\label{eq-v-n}
\|v^n_{\theta_n}\|_{\mathcal{D}^{k}_{0}}\le \begin{cases}
C_\rho \|v^n\|_{\mathcal{D}^{k}_{0}}\le {C}_3\epsilon\, \theta_n^{\max\{0,\,k+2-\tilde{k}\}},
\quad 0\le k\le k_0-1,\\[2mm]
C_\rho \theta_n \|v^n\|_{\mathcal{D}^{k_0-1}_{0}}\le
{C}_3\epsilon\, \theta_n^{\max\{1,\,k_0+2-\tilde{k}\}}, \quad
 k=k_0,
\end{cases}
\end{equation}
hold, where $C_\rho>0$ is given in \eqref{f2.5}.
\end{lemm}

\begin{proof}
From the identity
$$
u^n-u^s=\tilde{u}^0+\sum\limits_{j=0}^{n-1}\delta u^j
$$
we have immediately by using \eqref{u-a-2} and Lemma \ref{lemm-delta-u-j}
that
\begin{equation}
\label{esti-u-u-s}
\begin{array}{ll}
\|u^n-u^s\|_{\mathcal{A}_{l}^{k}} & \le
\|\tilde{u}^0\|_{\mathcal{A}_{l}^{k}}+\sum\limits_{j=0}^{n-1}
\|\delta u^j\|_{\mathcal{A}_{l}^{k}}\\
& \le C^a\epsilon+C_1\epsilon\sum\limits_{j=0}^{n-1}
 \theta_j^{\max\{3-\tilde{k},\,k-\tilde{k}\}}\Delta \theta_j\\
& \le C^a\epsilon+C_1\tilde{C}\epsilon \theta_n^{\max\{0,\,k+1-\tilde{k}\}}.
\end{array}
\end{equation}
Here, we have used the fact that
\begin{equation}\label{ineq}
\sum_{p=0}^{j-1} \theta_p^{k-\tilde{k}}\Delta \theta_p
\le \begin{cases}
\tilde{C} \theta_j^{k+1-\tilde{k}}, \quad {\rm as}~k-\tilde{k}\ge 0,\\
\tilde{C}, \quad {\rm as}~k-\tilde{k}\le -2,
\end{cases}
\end{equation}
for an absolute constant $\tilde{C}$.

From \eqref{esti-u-u-s}, we obtain the estimate \eqref{esti-u-n} immediately.
Similarly, from the identity
$$
v^n={v}^0+\sum\limits_{j=0}^{n-1}\delta v^j,
$$
we can easily deduce the estimates \eqref{L-infinity} and \eqref{eq-v-n}  by using Lemma \ref{lemm-delta-v-j}.
\end{proof}

As a direct consequence of the estimate \eqref{esti-u-n},  there is a
constant $\tilde{C}_3>0$ such that
\begin{equation}\label{esti-un-theta}
\|u^n_{\theta_n}-u^s\|_{\mathcal{A}_{\ell}^{k}}\le
\tilde{C}_3\epsilon \theta^{\max\{0,\, k+1-\tilde{k}\}}_n,\qquad
0\le k\le k_0.
\end{equation}

To get the estimate of $\lambda_k^n$, we need to estimate the norms of $\eta^n=\frac{\partial^2_y u^n_{\theta_n}}
{\partial_y u^n_{\theta_n}}$ and
\begin{equation}\label{zeta}
\zeta^n=\frac{(\partial_t+  u^n_{\theta_n}\partial_x +
v^n_{\theta_n}\partial_y-\partial^2_y)\partial_y u^n_{\theta_n}}
{\partial_yu_{\theta_n}^n},
\end{equation}
which are given as follows. Again, the proofs of the next two Lemmas will be given in  Section  \ref{section7}.

\begin{lemm} \label{lemm-eta-n}
Under the same assumptions as for Lemma \ref{lemm-delta-u-j}, there is a constant $C_4>0$, such that
\begin{equation}\label{esti-eta-n}
\|\eta^n-\bar\eta^n\|_{\mathcal{A}_{\ell}^{k}}
\le
\begin{cases}
C_4\epsilon\, \theta_n^{\max\{1, k+2-\tilde{k}\}},\qquad 4\le k\le k_0,\\[2mm]
C_4\epsilon, \qquad k= 3,\end{cases}\end{equation}
and
\begin{equation}\label{esti-eta-n-1}
\|\bar\eta^n\|_{{\mathcal C}_{\ell}^{k}}
\le {C}_4(1+\epsilon\theta_n^{\max\{0,\, k+3-\tilde{k}\}}), \qquad 0\le k\le k_0,\end{equation}
where $\bar\eta^n=\frac{\partial^2_y u^s}
{\partial_y u^n_{\theta_n}}$.
\end{lemm}

\begin{lemm}\label{lemm-zeta}
Under the same assumptions as for Lemma \ref{lemm-delta-u-j},
for $\zeta^n$ defined in \eqref{zeta}, there is a positive constant ${C}_5$ such that
\begin{align}\label{eq-zeta}
\|\zeta^n\|_{\mathcal{A}_{\ell}^{k}}\le
\begin{cases}
{C}_5\theta_n^{\max\{1,\, k+3-\tilde{k}\}},\qquad 4\le k\le k_0,\\[2mm]
{C}_5, \qquad k=3.\end{cases}
\end{align}
\end{lemm}

By plugging the estimates \eqref{eq-v-n},  \eqref{esti-eta-n}, \eqref{esti-eta-n-1},
\eqref{eq-zeta} and \eqref{esti-un-theta} into the definition \eqref{lambda-n}
of $\lambda_k^n$, we conclude

\begin{prop}\label{prop-lambda-n}
Suppose that the main assumption (MA), and
\eqref{wn} for  $w^j$, $0\le j\le n-1$, hold. There exists a positive constant $C_6>0$
depending on $C_p ~(1\le p\le 5)$ given in Lemmas \ref{lemm-delta-u-j}-\ref{lemm-zeta}, such that
$$
\lambda_k^n\le \begin{cases}
C_6\theta_n^{\max\{1,\, k+3-\tilde{k}\}},\qquad 4\le k\le k_0,\\[2mm]
C_{6},\qquad k= 3.\end{cases}
$$
\end{prop}

To estimate $\tilde{f}^n$ defined in \eqref{tilde-f-n}, we will need the following two
estimates whose proofs  will be given in  Section \ref{section7}.

\begin{lemm}\label{lemm-inverse}
Under the same assumptions as for Lemma \ref{lemm-delta-u-j}, there is a constant $C_7>0$, such that
\begin{equation}\label{esti-inverse}
\|\Big(\frac{\partial_yu^n_{\theta_n}}
{\partial_y u^s}\Big)^{-1}\|_{L^\infty}\le
2, \qquad
\|\Big(\frac{\partial_yu^n_{\theta_n}}
{\partial_y u^s}\Big)^{-1}\|_{\dot{\cA}^k_0}\le
C_7\epsilon\, \theta_n^{\max\{0, k+1-\tilde{k}\}},
\end{equation}
hold, with $1\le k\le k_0$.
\end{lemm}

\begin{lemm}\label{lemm-e-1}
Under the same assumptions as for Lemma \ref{lemm-delta-u-j},  there is a constant $C_8>0$, such
that for the error terms
$
e^{(1)}_j=\delta u^j\partial_x\delta u^j+\delta v^j\partial_y \delta u^j$ and
$$\begin{array}{ll}
e_{j}^{(2)}=& \big((1-S_{\theta_j}) (u^j-u^s)\big) \partial_x(\delta u^j)+
\delta u^j \partial_x \big((1-S_{\theta_j}) (u^j-u^s)\big)\\[2mm]
&+ \delta v^j \partial_y \big((1-S_{\theta_j}) (u^j-u^s)\big)
+\big((1-S_{\theta_n}) v^j\big) \partial_y(\delta u^j),
\end{array}$$
the following estimates
\begin{equation}\label{esti-e-1}
\|\frac{e^{(1)}_j}{\partial_yu^s}\|_{\mathcal{A}_{\ell}^{k_1}}
\le C_{8}\epsilon^2\theta_j^{\max\{6-2\tilde{k},\, k_1+3-2\tilde{k}\}}\Delta\theta_j,
\end{equation}
and
\begin{align}\label{esti-e-2}
\|\frac{e^{(2)}_j}{\partial_y u^s}\|_{\mathcal{A}_{l}^{k_1}}
&\le C_{8}
\epsilon^2 \theta_j^{\max(3-\tilde{k},k_1+5-2\tilde{k})}\Delta \theta_j,
\end{align}
hold for all $k_1\le k_0-1$ and $0\le j\le n-1$.
\end{lemm}

In summary, for
$$
\tilde{f}^n=\frac{f^n}{\partial_y u^n_{\theta_n}}=
 \frac{(S_{\theta_{n-1}}-S_{\theta_n})\big(\sum^{n-2}_{j=0} e_j
 \big)-S_{\theta_n} e_{n-1}+
(S_{\theta_{n-1}}-S_{\theta_n}){f}^a}{\partial_y u^n_{\theta_n}},
$$
with $e_{n}=e_{n}^{(1)}+e_{n}^{(2)}$, we have

\begin{prop}
Suppose that the main assumption (MA), and
\eqref{wn} for  $w^j$, $0\le j\le n-1$, hold, there exists a constant $C_9>0$ such that
\begin{equation}\label{esti-f-n}
\|\tilde{f}^n\|_{\mathcal{A}_{\ell}^{k}}
\le C_9\epsilon\theta_n^{\max\{3-\tilde{k},\,k-\tilde{k}\}}
\Delta \theta_n\,, \qquad 0\le k\le k_0.
\end{equation}
\end{prop}

\begin{proof} From
$\tilde{f}^n=\frac{f^n}{\partial_y u^n_{\theta_n}}
=\frac{f^n}{\partial_y u^s}\,\Big(\frac{\partial_y u^n_{\theta_n}}
{\partial_y u^s}\Big)^{-1}$, by using Lemma \ref{lemm2.1} and Lemma \ref{lemm-inverse} we have
\begin{align*}
\|\tilde{f}^n\|_{\mathcal{A}_{\ell}^{k}}
&\leq M_k\Big\{
2\|\frac{f^n}{\partial_y u^s}\|_{\mathcal{A}_{\ell}^{k}}+
\|\frac{f^n}{\partial_y u^s}\|_{L^\infty_{\ell}}
C_7\epsilon\, \theta_n^{\max\{0, k+1-\tilde{k}\}}\Big\}.
\end{align*}

On the other hand, using \eqref{f2.5} and \eqref{f2.5-3}, for any $k, k_j\ge 0$ ($j=1,2,3$), we have
\begin{align}\label{esti-f-n-1}
\|\frac{f^n}{\partial_y u^s}\|_{\mathcal{A}_{\ell}^{k}}
\le C_\rho \Big\{
\sum_{j=0}^{n-2} \|\frac{e_j}{\partial_yu^s}
\|_{\mathcal{A}_{\ell}^{k_1}} & \theta_n^{k-k_1}
\Delta \theta_n +\|\frac{e_{n-1}}
{\partial_yu^s}\|_{\mathcal{A}_{\ell}^{k_2}}\theta_n^{(k-k_2)_+}
\\
&+ \|\frac{f^a}{\partial_yu^s}\|_{\mathcal{A}_{\ell}^{
k_3}}\theta_n^{k-k_3}\Delta \theta_n\Big\}.\nonumber
\end{align}

Thus, by using \eqref{esti-e-1}, \eqref{esti-e-2} in \eqref{esti-f-n-1}, we get
\begin{align}
\|\frac{f^n}{\partial_y u^s}\|_{\mathcal{A}_{\ell}^{k}}
&\le C_\rho \Big\{2C_8
\sum_{j=0}^{n-2}\epsilon^2
\theta_j^{\max\{3-\tilde{k},\,k_1+5-2\tilde{k}\}}\Delta \theta_j\theta_n^{k-k_1}
\Delta \theta_n \notag\\
&+2C_8\epsilon^2
\theta_{n-1}^{\max\{3-\tilde{k},\,k_2+5-2\tilde{k}\}}\Delta \theta_{n-1}\theta_n^{(k-k_2)_+}
+C^a\epsilon\,\theta_n^{k-k_3}\Delta \theta_n\Big\},\label{esti-f-n-k}
\end{align}
for $k_1\le k_0-1$ and  $k_2\le k_0-1$, provided $\|\frac{f^a}{\partial_y u^s}
\|_{\mathcal{A}_{l}^{k_3}}\le C^a \epsilon$.

When $k=3$, by setting $k_1=k_3=\tilde{k}$ and $k_2=3$ in \eqref{esti-f-n-k} we get
\begin{align}
\|\frac{f^n}{\partial_y u^s}\|_{\mathcal{A}_{\ell}^{3}}
&\le C_\rho (2C_8(1+
\tilde{C})\epsilon^2+C^a\epsilon)
\theta_n^{3-\tilde{k}}\Delta \theta_n.\label{f-n-3}
\end{align}

When $4\le k\le k_0$, by choosing $k_1\ge 1+\tilde{k}$,
$k_2=\tilde{k}-2$ and $k_3=\tilde{k}$ in \eqref{esti-f-n-k} we obtain
\begin{align}
\|\frac{f^n}{\partial_y u^s}\|_{\mathcal{A}_{\ell}^{k}}
&\le C_\rho (2C_8(1+
\tilde{C})\epsilon^2+C^a\epsilon)
\theta_n^{k-\tilde{k}}\Delta \theta_n.\label{f-n-k}
\end{align}
Here, we have used the fact that $(k-k_2)_++k_2+5-\tilde{k}\le k$ for all $4\le k\le k_0$.

Combining \eqref{f-n-3} with \eqref{f-n-k}, we conclude the estimate \eqref{esti-f-n}.
\end{proof}

\smallskip
\noindent
\underline{\bf Proof of Theorem \ref{theo-wn}:}

We are now ready to conclude the proof of Theorem \ref{theo-wn} by induction on $n$.

For  $n=0$, from the main assumption (MA), Lemma \ref{u-a} and Lemma \ref{f-a}, we get immediately that
for any fixed $T>0$, there is a constant $C^a=C^a(k_0, T)$ such that
$$
\|\tilde{u}^0\|_{\cA_\ell^{k_0+1}}+
\|\frac{\partial_y\tilde{u}^0}
{\partial_yu^s}\|_{\cA_\ell^{k_0+1}}+
\|\frac{f^a}{\partial_yu^s}\|_{\cA_\ell^{k_0}}
\le C^a\epsilon.
$$
This implies that $\tilde{f}^0=\frac{f^a}{\partial_y(u^s+S_{\theta_0}\tilde{u}^0)}$ satisfies
$$
\|\tilde{f}^0\|_{\cA_\ell^{k_0}}
\le \tilde{C}^a\epsilon,
$$
for a constant ${\tilde C}^a$.

A direct calculation yields
$$
\lambda_k^0\le
\|u^s\|_{\mathcal{C}_{0}^{k}}
+\|\frac{\partial_y^2u^s}{\partial_yu^s}\|_{\mathcal{C}_{0}^{k}}
+\bar{C}^a\epsilon  \leq C_{k_0}, \qquad \forall k\le k_0,
$$
for a constant $\bar{C}^a$ depending on $C^a$ and $\tilde{C}^a$ given at above.

By applying Proposition \ref{prop-w-n} for $w^0$, and using the above estimates, it follows
\begin{equation}\label{esti-w-0}
\|w^0\|_{\cA_\ell^k}\le \bar{C}_{k_0} \epsilon, \qquad \forall k\le k_0,
\end{equation}
for a constant $\bar{C}_{k_0}$ depending on $\tilde{C}^a$ and $C_{k_0}$ given above.
Hence, the estimate \eqref{wn} for the case $n=0$ follows immediately from \eqref{esti-w-0}
with a constant $C_0$ depends on $\bar{C}_{k_0}, k_0, \tilde{k}$ and $\theta_0$.

Now, assuming that
\eqref{wn} holds for all $w^j$ with $0\le j\le n-1$, we are going to
prove it for $w^n$.
In fact, from the estimates \eqref{energy-A-1} and \eqref{esti-f-n}, we get
$$\|w^n\|_{\cA_\ell^k}\le C_1(\lambda_3^n)C_9\epsilon \theta_n^{\max\{3-\tilde{k}, k-\tilde{k}\}}
+C_2(\lambda_3^n)\lambda_k^n C_9\epsilon \theta_n^{3-\tilde{k}},$$
which implies by Proposition \ref{prop-lambda-n} that
$$\|w^n\|_{\cA_\ell^3}\le C_0\epsilon \theta_n^{3-\tilde{k}},$$
and
$$\|w^n\|_{\cA_\ell^k}\le C_0\epsilon\theta_n^{k-\tilde{k}},$$
for all $4\le k\le k_0$, with the constant $C_0\ge (C_1(C_6)+C_2(C_6)C_6)C_9$.

\smallskip
\subsection{Convergence of the iteration scheme}

In this subsection, we will prove the convergence of the iteration scheme
and this immediately yields the  existence of classical solutions to the Prandtl equation \eqref{eq-1}.

From the iteration scheme \eqref{u-n+1}-\eqref{prob-u-n} with $f^n$  defined in \eqref{fn-1},
we know that the approximate solution
$$
u^{n+1}=u^s+\tilde{u}^0+\sum\limits_{j=0}^n\delta u^j,\quad
v^{n+1}=v^0+\sum\limits_{j=0}^n\delta v^j,
$$
satisfies
\begin{equation}\label{appro}
\left\{\begin{array}{l}\mathcal{P}(u^{n+1}, v^{n+1})=
(1-S_{\theta_n})\sum^{n}_{j=0}e_j+S_{\theta_n} e_n +
(1-S_{\theta_n}) f^a,\\
\partial_x u^{n+1}+ \partial_y v^{n+1}=0,\\
u^{n+1}|_{y=0}=v^{n+1}|_{y=0}=0,\,\,\lim\limits_{y\to+\infty}
u^{n+1}=1,\\
u^{n+1}|_{t=0}=u_0(x, y)\,.
\end{array}\right.
\end{equation}

From the estimates \eqref{esti-delta-u} and \eqref{esti-delta-v}, we know that there
exist $u\in u^s+\cA_\ell^{\tilde{k}-2}$ and $v\in \cD_0^{\tilde{k}-3}$, such that
$$
\lim\limits_{n\to +\infty}\|u^n-u\|_{\cA_\ell^{\tilde{k}-2}}=0,\quad
\lim\limits_{n\to+\infty}\|v^n-v\|_{\cD_0^{\tilde{k}-3}}=0.
$$

To verify that the limit $(u,v)$ is a classical solution to the problem \eqref{eq-1},
it is enough to show that the right hand side of the equation in \eqref{appro} converges to zero as $n\to+\infty$.

Obviously, we have
$$
\|(1-S_{\theta_n})(f^a+\sum^{n}_{j=0}e_j)\|
_{\mathcal{A}_{\ell}^{k}}
\le \theta_n^{-1}(\|f^a\|_{\mathcal{A}_{\ell}^{k+1}}+
\sum^{n}_{j=0}\|e_j\|_{\mathcal{A}_{\ell}^{k+1}}).
$$
Thus, it is enough to prove the convergence of series
$
\sum^{+\infty}_{j=0}\|e_j\|_{\mathcal{A}_{\ell}^{k+1}}.
$
Recall
$$
e_{j}=e_{j}^{(1)}+e_{j}^{(2)},
$$
with
$$
e^{(1)}_j=\delta u^j\partial_x(\delta u^j)+\delta v^j\partial_y (\delta u^j),
$$
and
\begin{align*}
e_{j}^{(2)}=\partial_y\Big(\delta v^j  \big((1-S_{\theta_j})(u^j-u^s)\big)
+\big((1-S_{\theta_j}) v^j\big) (\delta u^j)\Big).
\end{align*}
By using Lemma \ref{lemm2.1}, it follows that
\begin{align*}
\|e_j^{(1)}\|_{\mathcal{A}_{\ell}^{k+1}}
&\le M_{k}\Big(\|\delta u^j\|_{L^\infty}\|\delta u^j\|_{\mathcal{A}_{\ell}^{k+2}}+
\|\delta v^j\|_{L^\infty} \|\delta u^j\|_{\mathcal{A}_{\ell}^{k+2}}\\
&\qquad+\|\delta v^j\|_{\mathcal{D}_{0}^{k+2}}
\|\delta u^j\|_{L^2_{y, \ell}(L^\infty_{t, x})}\Big)\\
&\le C_{10}\epsilon^2\,\theta_j^{k+5-2\tilde{k}}
\Delta\theta_j,
\end{align*}
 and
\begin{align*}
\|e_j^{(2)}\|_{\mathcal{A}_{\ell}^{k+1}}
&\le M_{k}\Big(\|\delta v^j\|_{L^\infty}\|u^j-u^s\|_{\mathcal{A}_{\ell}^{k+2}}
+\|\delta v^j\|_{\mathcal{D}_{0}^{k+2}}\|u^j-u^s\|_{L^2_{y,\ell}(L^\infty_{t, x})}\\
&+\|v^j\|_{L^\infty}\|\delta u^j\|_{\mathcal{A}_{\ell}^{k+2}}
+\|v^j\|_{\mathcal{D}_{0}^{k+2}}\|\delta u^j\|_{L^2_{y,\ell}(L^\infty_{t, x})}\Big)\\
&\le {C}_{10}\epsilon^2\,\theta_j^{k+3-\tilde{k}}
\Delta\theta_j,
\end{align*}
for a positive constant ${C}_{10}>0$, where we have used
 \eqref{esti-delta-u}, \eqref{esti-delta-v} and \eqref{esti-u-n}. Therefore, we obtain
$$
\sum^{+\infty}_{j=0}\|e_j\|_{\mathcal{A}_{\ell}^{k+1}}\le
{C}\sum^{+\infty}_{j=0}\theta_j^{k+3-\tilde{k}}
\Delta\theta_j\le {C}\tilde{C},
$$
for all $k\le \tilde{k}-5$. And this concludes the convergence of the iteration scheme
and the existence of classical solutions to the Prandtl equation \eqref{eq-1}.

\section{Uniqueness and stability}\label{section6}
\setcounter{equation}{0}

In this section, we study the stability of classical solutions to the Prandtl equation \eqref{eq-1},
and thus the uniqueness of the classical solution obtained in Section \ref{section5} will follow immediately.

Let $(u^1, v^1)$ and $(u^2, v^2)$ be two classical solutions to the problem \eqref{eq-1}
in the solution spaces given in Theorem \ref{theo1} with the initial data $u_0^1(x,y)$
and $u_0^2(x,y)$ as two small perturbations of $u_0^s(y)$ as stated in Theorem \ref{theo1}.
Denoting by
$$
u=u^1-u^2,\quad v=v^1-v^2,\quad \tilde{u}=\frac{u^1+u^2}{2},\quad
\tilde{v}=\frac{v^1+v^2}{2},
$$
then from \eqref{eq-1}, we deduce that $(u, v)$ satisfies
\begin{equation}\label{eq6.1}
\left\{\begin{array}{l} \partial_t u + \tilde{u} \partial_xu+\tilde{v}\partial_yu
+u\partial_x \tilde{u}
+v\partial_y \tilde{u}
-\partial^2_y u =0, \\
\partial_x u+\partial_y v=0,\\
u|_{y=0} =v|_{y=0}=0, \quad \lim\limits_{y\to+\infty} u =0, \\
u|_{t=0} =u_0(x,y)\, :=u_0^1-u_0^2\,.
\end{array}\right.
\end{equation}

As in  Section \ref{s3}, set
\begin{equation}\label{unknown-1}
w(t,x,y) = \left( \frac{u}{\partial_y \tilde{u}} \right)_y(t, x, y)
\quad \mbox{that is,}\quad
u(t, x, y)=(\partial_y \tilde{u}) \int^y_0 w(t,x, y')d y'.
\end{equation}

Then, from \eqref{eq6.1} we know that $w(t,x,y)$ satisfies
\begin{equation}\label{eq6.3}
\left\{\begin{array}{l}
\partial_t w + \partial_x (\tilde{u} w)+\partial_y(\tilde{v} w)-2
\partial_y (\eta w)
+\partial_y(\zeta\int^y_0 w(t, x, \tilde{y})d\tilde{y})-\partial^2_y w
=0,\\
\big(\partial_y w+2\eta w\big) |_{y=0} = 0, \\
w|_{t=0} =w_0(x,y)=\left( \frac{u_0}{\partial_y \tilde{u}} \right)_y(x, y), \end{array}\right.
\end{equation}
where
$$
\eta=\frac{\partial^2_y \tilde{u}}{\partial_y \tilde{u}},\quad \zeta
=\frac{\big(\partial_t+\tilde{u} \partial_x+
\tilde{v} \partial_y -\partial^2_y\big) \partial_y \tilde{u}}
{\partial_y \tilde{u}}.
$$

Similar to the  proof for \eqref{energy-A} in the problem \eqref{eq6.3}, it follows
\begin{equation}\label{esti-w-A}
\|w\|_{\mathcal{A}^k_{\ell}([0,T]\times \RR^2_+)} \le C(T)\|w_0\|_{\mathcal{A}^{k}_{\ell}(\RR^2_+)},
\qquad k\le \tilde{k}-3,
\end{equation}
for a constant $C(T)$ depending on $T>0$ and the norms of the initial data $u_0^1, u_0^2$ in the spaces
given in the existence part of Theorem \ref{theo1}.

From \eqref{esti-w-A}, and the transformation \eqref{unknown-1}, we deduce
$$
\|u^1-u^2\|_{\cA_\ell^{k}([0,T]\times \RR_+^2)}+\|v^1-v^2\|_{ \cD_0^{k-1}([0,T]\times\RR_+^2)}\le C\|\frac{\partial }{\partial y}\left(\frac{u^1_0-u_0^2}{\partial_y u_0^s}\right)\|_{\cA_\ell^k(\RR^2_+)},
$$
for all $k\le \tilde{k}-3$. And this concludes the uniqueness and stability  results
stated in Theorem \ref{theo1}.

\section{Proof of some technical estimates}\label{section7}
\setcounter{equation}{0}

Finally, in this section, we give
the proofs for Lemmas \ref{lemm-delta-u-j}, \ref{lemm-eta-n}, \ref{lemm-zeta}, \ref{lemm-inverse}
and \ref{lemm-e-1} stated in  Section \ref{section5} about the iteration scheme \eqref{u-n+1}-\eqref{prob-u-n}.

We start with the proof for Lemma \ref{lemm-delta-u-j}.

\smallskip
\noindent\underline{\bf Proof of Lemma \ref{lemm-delta-u-j}:}
Let us first prove the estimate \eqref{esti-delta-u}. First
of all, it holds true for $\delta u^0$. Indeed, from
$$
\delta u^0=\partial_y (u^s+S_{\theta_0}{\tilde u}^0)\int_0^y w^0 d\tilde{y},
$$
by using Lemma \ref{lemm2.1} and the Sobolev embedding theorem, we have
\begin{align*}
\|\delta u^0\|_{\mathcal{A}_{l}^{k}}&\le
C^0_k\|w^0\|_{\mathcal{A}_{l}^{k}}\le C^0_k C_0
\epsilon \theta_0^{\max\{3-\tilde{k},\,k-\tilde{k}\}}\Delta \theta_0, \qquad k\ge 2,
\end{align*}
with
\begin{equation*}
C^0_k=\|u^s\|_{\mathcal{C}^{k+1}_{\ell}}+\|\tilde{u}^0\|_{\mathcal{A}^{k+1}_{\ell}}.
\end{equation*}
The estimate \eqref{esti-delta-u} holds obviously for $\delta u^0$ when $k=0,1$.

Now, suppose that \eqref{esti-delta-u} holds for  $\delta u^{p}$,
 $0\le p\le j-1$. We estimate $\delta u^{j}$ as follows.

Recalling $w^j=(\frac{\delta u^j}{\partial_y (u_{\theta_j}^j)})_y$, we have
$$
\delta u^j=\partial_y( u_{\theta_j}^j)\int_0^y w^j d\tilde{y}.
$$

Using again Lemma \ref{lemm2.1}, the Sobolev embedding theorem and \eqref{ineq},
it follows that, for $k\ge 4$ and $\tilde{k}\ge 6$,
\begin{equation*}
\begin{array}{lll}
\|\delta u^j\|_{\mathcal{A}_{l}^{k}}&\le & M_k((C^0_k+\|\partial_y(u_{\theta_j}^j-u^s)
\|_{\mathcal{A}_{l}^{k}})
\|w^j\|_{\mathcal{A}_{l}^{2}}+\|\partial_yu_{\theta_j}^j
\|_{L^2_{y,\ell}(L^\infty_{t,x})}
\|w^j\|_{\mathcal{A}_{l}^{k}})\\
&\le & M_k((C^0_k+ C_\rho\theta_j
\sum\limits_{p=0}^{j-1}\|\delta u^p \|_{\mathcal{A}_{l}^{k}})
\|w^j\|_{\mathcal{A}_{l}^{2}}
+\|\partial_yu^j\|_{L^2_{y,\ell}(L^\infty_{t,x})}\|w^j
\|_{\mathcal{A}_{l}^{k}})\\
&\le & C_0M_kC_\rho(C_{k}^0+C_1\epsilon \, \theta_j\sum^{j-1}_{p=0}
\theta_p^{\max\{3-\tilde{k},\,k-\tilde{k}\}}\Delta \theta_p)\epsilon \theta_j^{3-\tilde{k}}\Delta \theta_j
\\
& &
+C_0C_\rho M_k(C_{3}^0+C_1\epsilon\sum^{j-1}_{p=0}
\theta_p^{3-\tilde{k}}\Delta \theta_p)\epsilon
\theta_j^{\max\{3-\tilde{k},\,k-\tilde{k}\}}\Delta \theta_j\\
&\le & \epsilon C_0C_\rho M_k(C_{k}^0+C_1\tilde{C}\epsilon\,
\theta_j^{1+\max\{0,\,k+1-\tilde{k}\}}
) \theta_j^{3-\tilde{k}}\Delta \theta_j\\
& &
+\epsilon C_0C_\rho M_k(C_{3}^0
+C_1\tilde{C}\epsilon)
\theta_j^{\max\{3-\tilde{k},\,k-\tilde{k}\}}\Delta \theta_j\\
&\le & \epsilon C_0C_\rho M_k(C_{k}^0\theta_0^{-1}+C^0_3+2C_1\tilde{C}\epsilon)
\theta_j^{k-\tilde{k}}\Delta \theta_j,\end{array}
\end{equation*}
where the constant $C_\rho$ comes from \eqref{f2.5}.

By setting $C_1= 4C_0C_\rho M_kC^0_3$, we can choose
$0<\epsilon\le \epsilon_0$ and $\theta_0>0$ large enough such that
\begin{equation*}
C_0C_\rho M_k(C_{k}^0\theta_0^{-1}+C^0_3+2C_1\tilde{C}\epsilon)\le C_1.
\end{equation*}
Therefore,  we get
\begin{equation*}
\|\delta u^j\|_{\mathcal{A}_{l}^{k}}\le
C_1\epsilon \theta_j^{ k-\tilde{k}}\Delta \theta_j,
 \end{equation*}
for $k\ge 4$.
On the other hand, we have
\begin{equation*}
\begin{array}{lll}
\|\delta u^j\|_{\mathcal{A}_{\ell}^{3}}
&\le & M_k(C^0_3+ C_\rho
\sum\limits_{p=0}^{j-1}\|\delta u^p \|_{\mathcal{A}_{\ell}^{4}})
\|w^j\|_{\mathcal{A}_{\ell}^{2}}
+M_k\|\partial_yu^j\|_{L^2_{y,\ell}(L^\infty_{t,x})}\|w^j
\|_{\mathcal{A}_{\ell}^{3}}\\
&\le & C_0M_kC_\rho(C_{3}^0+C_1\epsilon \, \sum^{j-1}_{p=0}
\theta_p^{4-\tilde{k}}\Delta \theta_p)\epsilon \theta_j^{3-\tilde{k}}\Delta \theta_j
\\
& &
+C_0C_\rho M_k(C_{3}^0+C_1\epsilon\sum^{j-1}_{p=0}
\theta_p^{3-\tilde{k}}\Delta \theta_p)\epsilon
\theta_j^{3-\tilde{k}}\Delta \theta_j\\
&\le & C_1\epsilon \theta_j^{3-\tilde{k}}\Delta \theta_j,\end{array}
\end{equation*}
for $\tilde{k}\ge 6$, by choosing a proper constant $C_1>0$.
And this completes the proof of the estimate \eqref{esti-delta-u}.

We now turn to  the estimates \eqref{Linfty} and \eqref{D-k-1}.
When  $j=0$, from
\begin{equation}\label{u-0}
\frac{\delta u^0}{\partial_y u^s}=(1+\frac{\partial_yS_{\theta_0}
(\tilde{u}^0)}{\partial_y u^s})\int_0^y
w^0(t,x,\tilde{y})d\tilde{y},
\end{equation}
we have, by using the Sobolev embedding theorem and for some $\ell>1/2$,  that
$$
\begin{array}{ll}
\|\frac{\delta u^0}{\partial_y u^s}\|_{L^\infty}& \le (1+
\|\frac{\partial_yS_{\theta_0}(\tilde{u}^0)}{\partial_y u^s}
\|_{L^\infty})\bar{C}_\ell\|
w^0\|_{L^2_{y,\ell}(L^\infty_{t,x})}\\
& \le(1+\|\frac{\partial_y
S_{\theta_0}(\tilde{u}^0)}{\partial_y u^s}\|_{L^\infty})
\bar{C}_\ell \|w^0\|_{\cA ^{2}_{\ell}},\end{array}
$$
where $\bar{C}_\ell=(\int_0^{+\infty}(1+y^2)^{-\ell}dy)^{\frac 12}$.
The estimate \eqref{Linfty} with $j=0$ follows immediately by choosing
\begin{equation*}
C_1\ge C_0\bar{C}_\ell(1+\|\frac{\partial_yS_{\theta_0}(\tilde{u}^0)}
{\partial_y u^s}\|_{L^\infty}).
\end{equation*}

Applying Lemma \ref{lemm2.1} to \eqref{u-0} gives
\begin{align*}
&\sum\limits_{k_1+[\frac{k_2+1}{2}]\le k-1}\|
\partial_{\mathcal T}^{k_1}\partial_y^{k_2}
\left(\frac{\delta u^0}{\partial_y u^s}\right)
\|_{L^\infty_y(L^2_{t,x})}\\
&\hspace{.1in}\leq M_k\bar{C}_\ell\left\{ \sum\limits_{1\leq k_1+
[\frac{k_2+1}{2}]\le k-1}\|\partial_{\mathcal T}^{k_1}
\partial_y^{k_2}\left(\frac{\partial_yS_{\theta_0}
(\tilde{u}^0)}{\partial_y u^s}\right)\|_{L^\infty}\|w^0\|_{L^2_\ell}\right.\\
&\hspace{.7in}\left.+(1+\|\frac{\partial_yS_{\theta_0}(\tilde{u}^0)}
{\partial_y u^s}\|_{L^\infty})\|w^0\|_{\cA ^{k-1}_{\ell}}\right\}.
\end{align*}
Then this yields the estimate \eqref{D-k-1} with $j=0$  by choosing
\begin{equation*}
C_1\ge M_kC_0\bar{C}_\ell\Big\{1+ \sum\limits_{0\leq k_1+
[\frac{k_2+1}{2}]\le k-1}\|\partial_{\mathcal T}^{k_1}
\partial_y^{k_2}\left((1+\frac{\partial_yS_{\theta_0}
(\tilde{u}^0)}{\partial_y u^s})\right)\|_{L^\infty}
\Big\}.
\end{equation*}

When  $1\le j\le n-1$, by definition, we have
\begin{equation}\label{delta-uj}
\frac{\delta u^j}{\partial_y u^s}=\frac{\partial_yu^j_{\theta_j}}
{\partial_y u^s}\int_0^y w^j(t,x,\tilde{y})d\tilde{y}.
\end{equation}
By using Lemma \ref{lemm2.1} and the assumptions on $w^j$, it
is sufficient to obtain the bounds of
$$
\frac{\partial_yu^j_{\theta_j}}{\partial_y u^s}=1+\frac{\partial_y
S_{\theta_j}(\tilde{u}^0)}{\partial_y u^s}+\sum\limits_{p=0}^{j-1}
\frac{\partial_y S_{\theta_j}\delta u^p}{\partial_y u^s}
$$
in the spaces $L^\infty$ and $\cD_0^{k-1}$ respectively.

To obtain the estimate \eqref{Linfty} with index $j$,
suppose  that \eqref{Linfty} holds for
$\frac{\delta u^p}{\partial_y u^s}$ with $0\leq p\le j-1$
and prove it by induction. Note that $j=0$ holds by the above
argument.

For this, first we show

\begin{lemm}\label{lemm5001}
Suppose that the estimate \eqref{wn} holds for  $w^j$,
$0\le j\le n-1$, with $\tilde{k}\ge 7$,
then there exists a constant $M_s$, such that for all $0\le j\le n$,
\begin{align}
\|\frac{\partial_yu^j_{\theta_j}}{\partial_y u^s}\|_{L^\infty}\leq M_s
\label{eq-5001}
\end{align}
\end{lemm}

For continuity of the presentation, we postpone the proof of Lemma \ref{lemm5001}
later. Then by Lemma \ref{lemm5001}, with
 \eqref{delta-uj}, it follows
$$
\|\frac{\delta u^j}{\partial_y u^s}\|_{L^\infty}\le M_s\bar{C}_\ell C_0
\epsilon\,\theta^{3-\tilde{k}}_j\Delta\theta_j\leq C_1\epsilon\,
\theta^{3-\tilde{k}}_j\Delta\theta_j,
$$
with
\begin{equation*}
C_1\ge M_s\bar{C}_\ell C_0.
\end{equation*}

To prove \eqref{D-k-1} with index $j$, we also
suppose that it holds for
$\frac{\delta u^p}{\partial_y u^s}$, $0\leq p\le j-1$.

Then, obviously, we have the identity
\begin{equation}\label{coef}
\begin{array}{l}
\sum\limits_{p=0}^{j-1}\frac{\partial_y S_{\theta_j}\delta u^p}
{\partial_y u^s} =
\sum\limits_{p=0}^{j-1}\left\{S_{\theta_j}(\frac{\partial_y(\delta u^p)}
{\partial_y u^s})
+[\frac{1}{\partial_y u^s}, S_{\theta_j}](\partial_y(\delta u^p))\right\}\\
\quad = \sum\limits_{p=0}^{j-1}\left\{S_{\theta_j}(\partial_y\frac{\delta u^p}
{\partial_y u^s})
+S_{\theta_j}(\frac{\delta u^p}{\partial_y u^s}\,\frac{ \partial^2_y u^s}
{\partial_y u^s})
+[\frac{1}{\partial_y u^s}, S_{\theta_j}](\partial_y(\delta u^p))\right\}.
\end{array}
\end{equation}

Thus, we have
\begin{align*}
&\|
\sum\limits_{p=0}^{j-1}\partial_{\mathcal T}^{k_1}\partial_y^{k_2}
\left(\frac{\partial_y S_{\theta_j}\delta u^p}{\partial_y u^s}\right)
\|_{L_y^\infty(L^2_{t,x})}\\
&\le  \tilde{B}\sum\limits_{p=0}^{j-1}\left\{
\|\partial_{\mathcal T}^{k_1}\partial_y^{k_2}S_{\theta_j}
\partial_y\Big(\frac{\delta u^p}{\partial_y u^s}
\Big)\|_{L_y^\infty(L^2_{t,x})}
+\|\partial_{\mathcal T}^{k_1}\partial_y^{k_2}
\Big(\frac{\delta u^p}{\partial_y u^s}\Big)\|_{L_y^\infty(L^2_{t,x})}\right\}
\\
&\le  \tilde{B}\sum\limits_{p=0}^{j-1}\left\{\theta_j\|
\partial_{\mathcal T}^{k_1}\partial_y^{k_2}
\Big(\frac{\delta u^p}{\partial_y u^s}\Big)\|_{L_y^\infty(L^2_{t,x})}
+\|\partial_{\mathcal T}^{k_1}\partial_y^{k_2}
\Big(\frac{\delta u^p}{\partial_y u^s}\Big)\|_{L_y^\infty(L^2_{t,x})}\right\},
\end{align*}
where the constant $\tilde{B}$ depends on the commutators
in \eqref{coef} which is
independent of $j$ and $p$.  Using the induction hypothesis for \eqref{D-k-1}, we have
\begin{align*}
&\sum\limits_{p=0}^{j-1}\sum\limits_{k_1+[\frac{k_2+1}{2}]\le k-1}
\|\partial_{\mathcal T}^{k_1}\partial_y^{k_2}\Big(\frac{\delta u^p}
{\partial_y u^s}\Big)\|_{L_y^\infty(L^2_{t,x})}\\
&\qquad\leq
C_1\epsilon\sum\limits_{p=0}^{j-1}\theta_p^{\max\{3-\tilde{k},
k-1-\tilde{k}\}}\Delta\theta_p
\leq C_1\tilde{C}\epsilon\,\theta_j^{\max\{0,
k-\tilde{k}\}}.
\end{align*}
Therefore, we deduce
\begin{align}\label{eq5.20}
&\sum\limits_{k_1+[\frac{k_2+1}{2}]\le k-1}\|\partial_{\mathcal T}^{k_1}
\partial_y^{k_2}\Big(\frac{\partial_y u^j_{\theta_j}}{\partial_y u^s}
\Big)\|_{L_y^\infty(L^2_{t,x})}\notag\\
&\hspace{.5in}\leq \sum\limits_{k_1+[\frac{k_2+1}{2}]\le k-1}\Big\{
\|\partial_{\mathcal T}^{k_1}\partial_y^{k_2}\left(\frac{\partial_y
S_{\theta_j}\tilde{u}^0}{\partial_y u^s}\right)\|_{L_y^\infty(L^2_{t,x})}
\notag\\
&\hspace{.5in}
\qquad\qquad+\|\sum\limits_{p=0}^{j-1}\partial_{\mathcal T}^{k_1}
\partial_y^{k_2}\left(\frac{\partial_y S_{\theta_j}\delta u^p}
{\partial_y u^s}\right)\|_{L_y^\infty(L^2_{t,x})}\Big\}\\
&\hspace{.5in}\leq C^a\epsilon+2\tilde{B}\tilde{C}C_1\epsilon \theta_j^{\max\{1,
k+1-\tilde{k}\}}.\notag
\end{align}
Now using Lemma \ref{lemm2.1}, it follows
\begin{align}\label{eq6.29}
&\sum\limits_{k_1+[\frac{k_2+1}{2}]\le k-1}\|\partial_{\mathcal T}^{k_1}
\partial_y^{k_2}\left(\frac{\delta u^j}{\partial_y u^s}\right)
\|_{L^\infty_y(L^2_{t,x})}\\
&\le
M_k\Big\{\sum\limits_{k_1+[\frac{k_2+1}{2}]\le k-1}\bar{C}_\ell
\Big\|\partial_{\mathcal T}^{k_1}\partial_y^{k_2}
\Big(\frac{\partial_y u^j_{\theta_j}}{\partial_y u^s}\Big)
\Big\|_{L^\infty_y(L^2_{t,x})}\|
w^j\|_{L^2_{y, \ell}(L^\infty_{t, x})}\notag\\
&+\sum\limits_{k_1+[\frac{k_2+1}{2}]\le k-1}\Big\|
\Big(\frac{\partial_y u^j_{\theta_j}}{\partial_y u^s}
\Big)\Big\|_{L^\infty}\Big\|\Big(\partial_{\mathcal T}^{k_1}
\partial_y^{k_2}\int_0^y
w^j(t,x,\tilde{y})d\tilde{y}\Big)\Big\|_{L^\infty_{y}(L^2_{t,x})}\Big\}\notag\\
&\le M_k\bar{C}_\ell\Big((C^a+2\tilde{B}\tilde{C}C_1
\theta_j^{\max\{1,\, k+1-\tilde{k}\}} )\epsilon\,
\|w^j\|_{{\mathcal A}^{2}_{\ell}}+
M_s\|w^j\|_{{\mathcal A}^{k-1}_{\ell}}\Big)\notag
\\
&\le C_1\epsilon\,\theta_j^{\max\{4-\tilde{k},\, k-1-\tilde{k}\}}\Delta\theta_j,\notag
\end{align}
by choosing $C_1$ satisfying
$$
C_1\ge M_kC_0\bar{C}_\ell ((C^a+2\tilde{B}\tilde{C}C_1)\epsilon+M_s).
$$
Then, we have proved \eqref{D-k-1} for $4\le k-1\le k_0-1$.

On the other hand, similar to the argument
to derive \eqref{eq5.20} from \eqref{coef}, one can deduce
$$
\sum\limits_{k_1+[\frac{k_2+1}{2}]\le 3}\|\partial_{\mathcal T}^{k_1}
\partial_y^{k_2}\Big(\frac{\partial_y u^j_{\theta_j}}{\partial_y u^s}
\Big)\|_{L_y^\infty(L^2_{t,x})}\le C\epsilon,
$$
for a positive constant $C>0$. Thus,  as in \eqref{eq6.29}, we obtain
\begin{equation*}
\|\frac{\delta u^j}{\partial_y u^s}
\|_{{\cD}^{3}_{0}}\le C_1\epsilon
\,\theta_j^{3-\tilde{k}}
\Delta\theta_j.
\end{equation*}
This completes the proof of Lemma \ref{lemm-delta-u-j}.

We now turn to the proof of Lemma \ref{lemm5001}.

\smallskip
\noindent\underline{\bf Proof of Lemma \ref{lemm5001}:}
The case when $0\le j\le 1$ is obvious. From \eqref{coef}, we have
\begin{align*}
&\|\sum\limits_{p=0}^{j-1}\frac{\partial_y S_{\theta_j}\delta u^p}
{\partial_y u^s}\|_{L^\infty}
\le  \sum\limits_{p=0}^{j-1}\Big\{\|S_{\theta_j}(\partial_y
\frac{\delta u^p}{\partial_y u^s})\|_{L^\infty}
\\
&\qquad+\|S_{\theta_j}(\delta u^p\frac{ \partial^2_y u^s}
{(\partial_y u^s)^2})\|_{L^\infty}
+\|[\frac{1}{\partial_y u^s}, S_{\theta_j}](\partial_y(\delta u^p))
\|_{L^\infty}\Big\}\\
&\le \sum\limits_{p=0}^{j-1}\Big\{\|S_{\theta_j}(\partial_y
\frac{\delta u^p}{\partial_y u^s})\|_{L^\infty}
+B
\|\frac{\delta u^p}{\partial_y u^s}\|_{L^\infty}\Big\},
\end{align*}
where we have used the estimate of commutators associated with the
mollifier $S_{\theta_j}$ given in Lemma \ref{lemm-commutators}.

Suppose now \eqref{eq-5001} holds for $0\leq p\leq j-1$. Let us check the case
when $p=j$. Obviously, we have
\begin{align}\label{smooth-AB}
&S_{\theta_j}(\partial_y\frac{\delta u^p}{\partial_y u^s})
= S_{\theta_j}\partial_y\left(\frac{\partial_yu^p_{\theta_p}}
{\partial_y u^s}\int_0^y
w^p(t,x,\tilde{y})d\tilde{y}\right)\\
&= S_{\theta_j}\left(\Big(\partial_y\frac{\partial_yu^p_{\theta_p}}
{\partial_y u^s}\Big)\int_0^y
w^p(t,x,\tilde{y})d\tilde{y}\right)+ S_{\theta_j}
\left(\frac{\partial_yu^p_{\theta_p}}{\partial_y u^s}
w^p\right).\notag
\end{align}
For the second term given on the right hand side
of \eqref{smooth-AB}, using the induction hypothesis, we get
\begin{align*}
&\sum\limits_{p=0}^{j-1}\|S_{\theta_j}(\frac{\partial_yu^p_{\theta_p}}
{\partial_y u^s}
w^p)\|_{L^\infty}\le\sum\limits_{p=0}^{j-1}\|\frac{\partial_y
u^p_{\theta_p}}{\partial_y u^s}
\|_{L^\infty}\|w^p\|_{L^\infty}\\
&\leq M_s\sum\limits_{p=0}^{j-1}\|w^p\|_{\cA^2_{0}}\le
M_sC_0\epsilon\sum\limits_{p=0}^{j-1}\theta_p^{3-\tilde{k}}
\Delta\theta_p\leq M_sC_0\tilde{C}\epsilon\,.
\end{align*}

For the first term on the right hand side of  \eqref{smooth-AB}, using Lemma
\ref{lemm-commutators} gives when $\tilde{k}\ge 7$,
\begin{align*}
&\sum\limits_{p=0}^{j-1}\|S_{\theta_j}\left(\Big(\partial_y
\frac{\partial_yu^p_{\theta_p}}
{\partial_y u^s}\Big)\int_0^y
w^p(t,x,\tilde{y})d\tilde{y}\right)\|_{L^\infty}\\
&=\sum\limits_{p=0}^{j-1}\|S_{\theta_j}\left\{\partial_y
\left(\frac{\partial_yS_{\theta_p}
\tilde{u}^0}{\partial_y u^s}+\sum\limits_{q=0}^{p-1}
\frac{\partial_y S_{\theta_p}\delta u^q}{\partial_y u^s}\right)\int_0^y
w^p(t,x,\tilde{y})d\tilde{y}\right\}\|_{L^\infty}\\
&\le \sum\limits_{p=0}^{j-1}\Big\{\bar{C}_\ell\|w^p\|_{L^2_{y,\ell}(L^\infty_{t,x})}
\Big(\|\partial_y\frac{\partial_yS_{\theta_p}\tilde{u}^0}
{\partial_y u^s}\|_{L^\infty}+\sum\limits_{q=0}^{p-1}
\big(\theta_p^2\|\frac{\delta u^q}{\partial_y u^s}\|_{L^\infty}\\
&\qquad+\|\partial_y[\frac{1}{\partial_y u^s}, \partial_yS_{\theta_p}]
\delta u^q\|_{L^\infty}
\big)\Big)\Big\}
\\
&\le \tilde{C}\bar{C}_\ell C_0\epsilon\,\sum\limits_{p=0}^{j-1}
\theta_p^{3-\tilde{k}}\Delta\theta_p \le \tilde{C}^2\bar{C}_\ell C_0\epsilon\,.
\end{align*}

In summary, we conclude
\begin{align}\label{eq5.17}
 &\|\frac{\partial_yu^j_{\theta_j}}{\partial_y u^s}\|_{L^\infty}\leq 1
 +\|\frac{\partial_y\tilde{u}^0_{\theta_j}}{\partial_y u^s}\|_{L^\infty}
 +\|\sum\limits_{p=0}^{j-1}\frac{\partial_y S_{\theta_j}\delta u^p}
 {\partial_y u^s}\|_{L^\infty}\\
 &\leq 1+C^a \epsilon+
 \tilde{C}^2\bar{C}_\ell C_0\epsilon+M_sC_0\tilde{C}\epsilon  +
 B C_1  \tilde{C}\epsilon,\notag
\end{align}
which implies the estimate \eqref{eq-5001} by choosing
$$
M_s\ge 2\Big(1+C^a\epsilon+\tilde{C}^2\bar{C}_\ell C_0\epsilon+B C_1\tilde{C}\epsilon\Big),
$$
and
$$
0<\epsilon\le  \frac1{2C_0 \tilde{C}}.
$$

\smallskip
To prove Lemma \ref{lemm-eta-n} for $\eta^n$, we need  the following
\begin{lemm}
Under the  assumptions of Lemma \ref{lemm5001},
there exists a constant $C_{11}$, such that for all $0\le j\le n$,
\begin{equation}
\|\frac{\partial_y(u^j_{\theta_j}-u^s)}{\partial_y u^s}\|_{\cA^k_{\ell}}
\leq C_{11}\epsilon
\theta_j^{\max\{0, k+1-\tilde{k}\}}.\label{eq-5002}
\end{equation}
\end{lemm}

{\it Proof:} This lemma can also be proved by induction on $j$. Suppose that it holds for
$0\le p\le j-1$, let us study the case when $p=j$.

From the identity
$$\frac{\partial_y(u^j_{\theta_j}-u^s)}{\partial_y u^s}=
\sum\limits_{p=0}^{j-1}
\frac{\partial_y S_{\theta_j}\delta u^p}{\partial_y u^s}+\frac{\partial_y S_{\theta_j}\tilde{u}^0}{\partial_y u^s},
$$
we have
\begin{align}
&\|\frac{\partial_y(u^j_{\theta_j}-u^s)}{\partial_y u^s}\|_{\cA^k_{\ell}}\le \|\sum\limits_{p=0}^{j-1}
\frac{\partial_y S_{\theta_j}\delta u^p}
{\partial_y u^s}\|_{\cA^k_{\ell}}+\|\frac{\partial_y S_{\theta_j}\tilde{u}^0}{\partial_y u^s}
\|_{\cA^k_{\ell}}\notag\\
&\qquad
\le  \sum\limits_{p=0}^{j-1}\Big\{\|S_{\theta_j}(\partial_y
\frac{\delta u^p}{\partial_y u^s})\|_{\cA^k_{\ell}}
+\|S_{\theta_j}(\delta u^p\frac{ \partial^2_y u^s}
{(\partial_y u^s)^2})\|_{\cA^k_{\ell}}\label{eq6.34}\\
&\qquad+\|[\frac{1}{\partial_y u^s}, S_{\theta_j}](\partial_y(\delta u^p))
\|_{\cA^k_{\ell}}\Big\}
+\|\frac{\partial_y S_{\theta_j}\tilde{u}^0}{\partial_y u^s}\|_{\cA^k_{\ell}}\notag\\
&\qquad\le \sum\limits_{p=0}^{j-1}\Big\{\|S_{\theta_j}(\partial_y
\frac{\delta u^p}{\partial_y u^s})\|_{\cA^k_{ \ell}}
+\tilde{B}
\|\frac{\delta u^p}{\partial_y u^s}\|_{\cA^k_{ \ell}}\Big\}+\|\frac{\partial_y S_{\theta_j}
\tilde{u}^0}{\partial_y u^s}\|_{\cA^k_{\ell}}\notag.
\end{align}

To estimate $\|S_{\theta_j}(\partial_y
\frac{\delta u^p}{\partial_y u^s})\|_{\cA^k_{\ell}}$, we use the relation \eqref{smooth-AB}.
For the second term on the right hand side of \eqref{smooth-AB}, by using \eqref{eq-5001} we have
\begin{align*}
&\sum\limits_{p=0}^{j-1}\|S_{\theta_j}(\frac{\partial_yu^p_{\theta_p}}
{\partial_y u^s}
w^p)\|_{\cA^k_{\ell}}\le  M_{k}
\sum\limits_{p=0}^{j-1}\Big\{\|\frac{\partial_y
u^p_{\theta_p}}{\partial_y u^s}
\|_{L^\infty}\|w^p\|_{\cA^{k}_{\ell}}
+\|\frac{\partial_y(u^p_{\theta_p}-u^s)}{\partial_y u^s}
\|_{\cA^{k}_{\ell}}\|w^p\|_{L^\infty}\Big\}
\\
&\le M_{k}\sum\limits_{p=0}^{j-1}\Big\{M_s C_0\epsilon\,\theta^{\max\{3-\tilde{k},\,k-\tilde{k}\}}_p
\Delta\theta_p+
C_{11}C_0\epsilon^2\theta^{\max\{3-\tilde{k},k+4-2\tilde{k}\}}_p\Delta\theta_p\Big\}\\
&\leq M_k \tilde{C}\,C_0\epsilon(M_s+C_{11}\epsilon)
\theta_j^{\max\{0, k+1-\tilde{k}\}}.
\end{align*}

For the first term on the right hand side of  \eqref{smooth-AB}, we have
 for $k\ge 4$,
\begin{align*}
&\sum\limits_{p=0}^{j-1}\|S_{\theta_j}\left(\Big(\partial_y
\frac{\partial_yu^p_{\theta_p}}
{\partial_y u^s}\Big)\int_0^y
w^p(t,x,\tilde{y})d\tilde{y}\right)\|_{\cA^k_{ \ell}}\\
&\le M_k\theta_j \bar{C}_\ell\sum\limits_{p=0}^{j-1}
\Big\{\|w^p\|_{L^2_{y,\ell}(L^\infty_{t,x})}
\|\partial_y\frac{\partial_y({u}^p_{\theta_p}-u^s)}
{\partial_y u^s}\|_{\cA^{k-1}_{ \ell}}
+\|w^p\|_{\cA^{k-1}_{ \ell}}
\|\partial_y\frac{\partial_y{u}^p_{\theta_p}}
{\partial_y u^s}\|_{L^2_{y, \ell}(L^\infty_{t, x})}\Big\}
\\
&\le M_k\theta_j \bar{C}_\ell\sum\limits_{p=0}^{j-1}\Big\{\|w^p\|_{\cA^2_{ \ell}}
\|\frac{\partial_y({u}^p_{\theta_p}-u^s)}
{\partial_y u^s}\|_{\cA^{k}_{ \ell}}
+\|w^p\|_{\cA^k_{\ell}}
\|\frac{\partial_y({u}^p_{\theta_p}-u^s)}
{\partial_y u^s}\|_{\cA^3_{\ell}}\Big\}
\\
&\le M_k\theta_j \bar{C}_\ell\sum\limits_{p=0}^{j-1}\Big\{C_0
\theta_p^{3-\tilde{k}}
\Delta\theta_pC_{11}\epsilon^2
\theta_p^{\max\{0, k+1-\tilde{k}\}}
+C_0\theta_p^{\max\{3-\tilde{k},\, k-\tilde{k}\}}
\Delta\theta_p
\,C_{11}\epsilon^2 \Big\}\\
&\le 2M_k\bar{C}_\ell C_0\tilde{C}C_{11} \epsilon^2 \theta_j^{\max\{1,\, k+2-\tilde{k}\}}\,,
\end{align*}
and
\begin{align*}
&\sum\limits_{p=0}^{j-1}\|S_{\theta_j}\left(\Big(\partial_y
\frac{\partial_yu^p_{\theta_p}}
{\partial_y u^s}\Big)\int_0^y
w^p(t,x,\tilde{y})d\tilde{y}\right)\|_{\cA^3_{\ell}}\le B_1 \epsilon^2 \,,
\end{align*}
for a positive constant $B_1$.

By plugging the above three estimates into the first term on the right hand side of \eqref{eq6.34},
\eqref{eq-5002} follows  by choosing a proper constant $C_{11}>0$.

\begin{rema}
From the above argument, it is easy to see that the estimates \eqref{eq-5001} and \eqref{eq-5002}
hold without mollifying $(\,\cdot\,)_{\theta_j}$, but for $k\le k_0-1$.
\end{rema}

We are now ready to prove Lemma \ref{lemm-inverse}.

\smallskip
\noindent\underline{\bf Proof of Lemma \ref{lemm-inverse}:} From the estimate \eqref{eq5.17},
we get immediately that there is $\epsilon_0>0$ such that when $0<\epsilon\le \epsilon_0$, it holds that
\begin{equation}\label{bound}
\|\frac{\partial_yu_{\theta_j}^j}{\partial_yu^s}\|_{L^\infty}\le 2, \quad \inf|\frac{\partial_y
u_{\theta_j}^j}{\partial_yu^s}|\ge \frac{1}{2},
\end{equation}
for  $0\le j\le n$. Hence,  the first estimate given in \eqref{esti-inverse} follows.

By using the following F\`a Di Bruno formula,
$$
\partial^m( g(f))=m!
\sum_{1\le r\le m}\frac{1}{r!} g^{(r)}(f)\prod_{m_1+\cdots m_r=m, m_j\ge 1}
\frac{1}{m_j!}\partial^{m_j}f,
$$
we have,
\begin{align*}
\|(\frac{\partial_y u^n_{\theta_n}}{\partial_y u^s})^{-1}\|_{\dot{\mathcal{A}}_{0}^{k}}
\leq& B_k \sum_{1\le k_1+\left[\frac{k_2+1}{2}\right]\leq k}\sum_{1\le r\le k_1+k_2}\|\prod\limits_{1\le j\le r}
\partial^{m^1_j}_\cT\partial^{m^2_j}_y\Big(\frac{\partial_y u^n_{\theta_n}}{\partial_y u^s}\Big)
\|_{L^2}\\
\leq& B_k \sum_{1\le k_1+\left[\frac{k_2+1}{2}\right]\leq k}\sum_{1\le r\le k_1+k_2}
\|
\partial^{m^1_r}_\cT\partial^{m^2_r}_y\Big(\frac{\partial_y (u^n_{\theta_n}-u^s)}{\partial_y u^s}\Big)
\|_{L^2}\\
&\qquad\times\prod\limits_{1\le j\le r-1}\|
\partial^{m^1_j}_\cT\partial^{m^2_j}_y\Big(\frac{\partial_y (u^n_{\theta_n}-u^s)}{\partial_y u^s}\Big)
\|_{L^\infty}\\
\leq& B_k \sum_{1\le k_1+\left[\frac{k_2+1}{2}\right]\leq k}\sum_{1\le r\le k_1+k_2}
\|\Big(\frac{\partial_y (u^n_{\theta_n}-u^s)}{\partial_y u^s}\Big)
\|_{\cA^{m^1_r+[\frac{m^2_r+1}{2}]}_0}\\
&\qquad\times\prod\limits_{1\le j\le r-1}\|\Big(\frac{\partial_y (u^n_{\theta_n}-u^s)}{\partial_y u^s}\Big)
\|_{\cA^{m^1_j+2+[\frac{m^2_j+2}{2}]}_0},
\end{align*}
where $m^1_1+\cdots m^1_r=k_1, m^2_1+\cdots m^2_r=k_2,
m^1_j+m^2_j\ge 1$, with $m_r^1$ and $m_r^2$ being supposed to be the largest integers in the corresponding
group of indices, respectively. Then, by using \eqref{eq-5002} in the above inequality, it follows
\begin{align}
 \|(\frac{\partial_y u^n_{\theta_n}}{\partial_y u^s})^{-1}\|_{\dot{\mathcal{A}}_{0}^{k}}
 \leq C_7\epsilon\, \theta_n^{\max\{0, k+1-\tilde{k}\}}, \label{bruno}
\end{align}
for a positive constant $C_7>0$.
This completes the proof of the lemma.

\smallskip
Now, we turn to  the estimates on $\eta^n=\frac{\partial^2_yu^n_{\theta_n}}
{\partial_yu^n_{\theta_n}}$ stated in Lemma \ref{lemm-eta-n}.

\smallskip
\noindent\underline{\bf Proof of Lemma \ref{lemm-eta-n}:} Obviously, we have
\begin{eqnarray}\label{6.34}
\|\eta^n-\bar\eta^n\|_{\mathcal{A}_{\ell}^{k}}&=
&\|\frac{\partial^2_y(u_{\theta_n} ^n-u^s)}{\partial_y
u_{\theta_n} ^n}\|_{\mathcal{A}_{\ell}^{k}}\\
&\le& M_k\left( \|\frac{\partial^2_y(u_{\theta_n}^n-u^s)}
{\partial_y u^s}\|_{\mathcal{A}_{\ell}^{k}}\|
(\frac{\partial_y u_{\theta_n}^n}{\partial_y u^s})^{-1}
\|_{L^\infty}\right.\nonumber
\\
&& \left.+\|\frac{\partial^2_{y}(u^n_{\theta_n} -u^s)}
{\partial_y u^s}\|_{L^\infty_\ell}\|\frac{\partial_y u^s}
{\partial_y u^n_{\theta_n}}\|_{\dot{\mathcal{A}}_{0}^{k}}
\right).\nonumber
\end{eqnarray}

We now estimate term by term on  the right hand side of \eqref{6.34}.
Since
$$
\frac{\partial^2_y( u_{\theta_n}^n-u^s)}{\partial_y u^s}=
\partial_y\Big(\frac{\partial_y( u_{\theta_n}^n-u^s)}{\partial_y u^s}\Big)+
\frac{\partial_y( u_{\theta_n}^n-u^s)}{\partial_y u^s}\frac{\partial^2_y u^s}{\partial_y u^s},
$$
by using \eqref{eq-5002}, we have for $k\ge 4$
\begin{align}
&\|\frac{\partial^2_y(u_{\theta_n}^n-u^s)}{\partial_y u^s}
\|_{\mathcal{A}_{\ell}^{k}}\nonumber\\
&\le M_k\|\frac{\partial^2_y u^s}{\partial_y u^s}\|_{{\cC}^{k}_{0}}
\|\frac{\partial_y( u_{\theta_n}^n-u^s)}{\partial_y u^s}
\|_{\mathcal{A}_{\ell}^{k}}+
\|\partial_y\Big(\frac{\partial_y( u_{\theta_n}^n-u^s)}{\partial_y u^s}\Big)
\|_{\mathcal{A}_{\ell}^{k}}\label{eq-5028}
\\
&\le M_k\|\frac{\partial^2_y u^s}{\partial_y u^s}\|_{{\cC}^{k}_{0}}
\|\frac{\partial_y( u_{\theta_n}^n-u^s)}{\partial_y u^s}
\|_{\mathcal{A}_{\ell}^{k}}+
C_\rho\theta_n\|\Big(\frac{\partial_y( u^n-u^s)}{\partial_y u^s}\Big)
\|_{\mathcal{A}_{\ell}^{k}}\nonumber\\
&\le \tilde{C}_{11}\epsilon\, \theta_n^{1+\max\{0, k+1-\tilde{k}\}}.\notag
\end{align}

On the other hand, by using \eqref{eq-5002}, it holds that
\begin{align}
\|\frac{\partial^2_y(u_{\theta_n}^n-u^s)}{\partial_y u^s}
\|_{\mathcal{A}_{\ell}^{3}}
&\le M_3\|\frac{\partial^2_y u^s}{\partial_y u^s}\|_{{\cC}^{3}_{0}}
\|\frac{\partial_y( u_{\theta_n}^n-u^s)}{\partial_y u^s}
\|_{\mathcal{A}_{\ell}^{3}}+
\|\frac{\partial_y( u_{\theta_n}^n-u^s)}{\partial_y u^s}
\|_{\mathcal{A}_{\ell}^{4}}\notag
\\
&\le (M_3\|\frac{\partial^2_y u^s}{\partial_y u^s}\|_{{\cC}^{3}_{0}}+1)C_{11}\epsilon.\label{eq-5028-3}
\end{align}

Plugging the estimates \eqref{eq-5028}, \eqref{eq-5028-3} and \eqref{bruno} into \eqref{6.34},
we obtain the estimate \eqref{esti-eta-n} given in Lemma \ref{lemm-eta-n}.

To derive the estimate \eqref{esti-eta-n-1}, from the definition of $\bar\eta^n$, we have
\begin{eqnarray*}
\|\bar\eta^n\|_{{\mathcal C}_{\ell}^{k}}
&\le& M_k\left( \|\frac{\partial^2_yu^s}{\partial_y u^s}
\|_{{\mathcal C}_{\ell}^{k}}\|\Big(\frac{\partial_y
 u^n_{\theta_n}}{\partial_y u^s}\Big)^{-1}\|_{L^\infty}+
\|\frac{\partial^2_{y}u^s}{\partial_y u^s}\|_{L^\infty_\ell}
\|\Big(\frac{\partial_y u^n_{\theta_n}}{\partial_y u^s}\Big)^{-1}
\|_{\dot{{\mathcal C}}_{0}^{k}}\right)\\
&\le& M_k\left( \|\frac{\partial^2_yu^s}{\partial_y u^s}
\|_{{\mathcal C}_{\ell}^{k}}\|\Big(\frac{\partial_y
 u^n_{\theta_n}}{\partial_y u^s}\Big)^{-1}\|_{L^\infty}+
\|\frac{\partial^2_{y}u^s}{\partial_y u^s}\|_{L^\infty_\ell}
\|\Big(\frac{\partial_y u^n_{\theta_n}}{\partial_y u^s}\Big)^{-1}
\|_{\dot{{\mathcal A}}_{0}^{k+2}}\right)\\
&\le &
 M_k(2\|\frac{\partial^2_yu^s}{\partial_y u^s}
\|_{{\mathcal C}_{\ell}^{k}} +\|\frac{\partial^2_{y}u^s}{\partial_y u^s}\|_{L^\infty_\ell}C_7\epsilon
\theta_n^{\max\{0,\, k+3-\tilde{k}\}}),
\end{eqnarray*}
where we have used \eqref{bruno} and \eqref{bound}.
Thus, we get the estimate \eqref{esti-eta-n-1} immediately. And this completes
the proof of the lemma.

To estimate $\zeta^n$, similar to the proof for Lemma \ref{lemm5001}, from
$$\begin{array}{lll}
\frac{\partial^2_{xy}S_{\theta_j}u^j}{\partial_y u^s}&=
&\frac{\partial^2_{xy} (S_{\theta_j}\tilde{u}^0)}{\partial_y u^s}
+\sum_{p=0}^{j-1} \left(S_{\theta_j}(\frac{\partial^2_{xy} \delta u^p}
{\partial_y u^s}) +[\frac{1}{\partial_y u^s}, S_{\theta_j}]
\partial_{x y}^2 \delta u^p\right)\\
&=& \frac{\partial^2_{xy} (S_{\theta_j}\tilde{u}^0)}{\partial_y u^s}
+\sum_{p=0}^{j-1} \left(S_{\theta_j}\left(\partial_x w^p+ \partial_x
(\frac{\delta u^p}{\partial_y u^s})\frac{\partial_y^2 u^s}{\partial_y u^s}\right)
+[\frac{1}{\partial_y u^s}, S_{\theta_j}]\partial_{x y}^2 \delta u^p\right),
\end{array}
$$
we have
\begin{lemm}\label{lemm6.11}
Under the assumptions of Lemma \ref{lemm5001}, there is a constant $C_{12}>0$, such that
$$
\|\frac{\partial_{xy}^2u_{\theta_j}^j}{\partial_yu_{\theta_j}^j}
\|_{\mathcal{A}_{\ell}^{k}}
\le C_{12}\epsilon \theta_j^{1+\max\{0, k+1-\tilde{k}\}},
$$
and
$$
\|\frac{\partial_{xy}^2u_{\theta_j}^j}{\partial_yu_{\theta_j}^j}
\|_{\mathcal{A}_{\ell}^{3}}
\le C_{12}\epsilon,
$$
hold for all $0\le j\le n$.
\end{lemm}

With this, we are now ready to prove Lemma \ref{lemm-zeta}.

\smallskip
\noindent\underline{\bf Proof of Lemma \ref{lemm-zeta}:}
First of all, note that
\begin{align*}
\|\frac{v^n_{\theta_n}\partial_y^2 u^n_{\theta_n}}
{\partial_y u^n_{\theta_n}}\|_{\mathcal{A}_{\ell}^{k}}
&\le  M_k( \|v^n_{\theta_n}\|_{L^\infty}\|\eta^n-\bar{\eta}^n
\|_{\mathcal{A}_{\ell}^{k}}+\|v^n_{\theta_n}\|_{L^\infty_y(L^2_{t,x})}
\|\bar{\eta}^n\|_{{\mathcal{C}}_{\ell}^{k}}\\
& \qquad+\|\eta^n\|_{L^2_{y,\ell}(L^\infty_{t,x})}
\|v^n_{\theta_n}
\|_{{\mathcal{D}}_{0}^{k}}).
\end{align*}
By using Lemma \ref{lemm-u-v-n} and Lemma \ref{lemm-eta-n}, we get
\begin{align*}
\|\frac{v^n_{\theta_n}\partial_y^2 u^n_{\theta_n}}
{\partial_y u^n_{\theta_n}}\|_{\mathcal{A}_{\ell}^{k}}
& \le
M_k\left(C_3C_4\epsilon\theta_n^{\max\{1,k+2-\tilde{k}\}}+
C_3C_4(1+\epsilon\theta_n^{\max\{0,k+3-\tilde{k}\}})\right.\\
& \qquad
+\left.2C_4{C}_3\epsilon \theta_n^{\max\{1,k+2-\tilde{k}\}}\right)\notag\\
& \le
C_5\theta_n^{\max\{1,k+3-\tilde{k}\}},\notag
\end{align*}
when $k\ge 4$ for a positive constant $C_5$. Moreover,
\begin{align*}
\|\frac{v^n_{\theta_n}\partial_y^2 u^n_{\theta_n}}
{\partial_y u^n_{\theta_n}}\|_{\mathcal{A}_{\ell}^3} &
\le
M_3(C_3C_4\epsilon+
C_3{C}_4(1+\epsilon)
+2C_3C_4\epsilon )\le C_8.
\end{align*}

Similarly, by using \eqref{sob-C} and Lemma \ref{lemm6.11}, we can show that
$$
\begin{array}{lll}
\|\frac{u^n_{\theta_n}\partial_{xy}^2 u^n_{\theta_n}}
{\partial_y u^n_{\theta_n}}\|_{\mathcal{A}_{\ell}^{k}}
&\le & M_k( \|u^n_{\theta_n}\|_{L^\infty}\|\frac{\partial_{xy}^2
u^n_{\theta_n}}{\partial_y u^n_{\theta_n}}\|_{\mathcal{A}_{\ell}^{k}}
 +\|\frac{\partial_{xy}^2 u^n_{\theta_n}}
 {\partial_y u^n_{\theta_n}}\|_{L^\infty}
 \| u^n_{\theta_n}-u^s\|_{{\mathcal{A}}_{\ell}^{k}}\\
 && \qquad +
 \|\frac{\partial_{xy}^2 u^n_{\theta_n}}
 {\partial_y u^n_{\theta_n}}\|_{L^\infty_{y}(L^2_{t,x})}
 \| u^s\|_{{\mathcal{C}}_{\ell}^{k}})\\
 &\le & M_k\left(C_s(C^a+C_3\epsilon)C_{12}\epsilon
 \theta_n^{1+\max\{0, k+1-\tilde{k}\}}\right.\\
 & &\qquad\left. +C_s(C^a+{C}_3\epsilon\theta_n^{\max\{0, k+1-\tilde{k}\}})C_{12}\epsilon\right)\\
 & \le &
C_5\theta_n^{\max\{1,k+2-\tilde{k}\}}.
\end{array}
$$
By noticing that
$$
\frac{(\partial_t-\partial^2_y)\partial_y u^n_{\theta_n}}{\partial_y u^n_{\theta_n}}
=\frac{(\partial_t-\partial^2_y)\partial_y u^n_{\theta_n}}{\partial_y u^s}
\cdot\frac{\partial_y u^s}{\partial_y u^n_{\theta_n}}=
\frac{(\partial_t-\partial^2_y)\partial_y
(u^n_{\theta_n}-{u}^s)}{\partial_y u^s}
\frac{\partial_y u^s}{\partial_y u^n_{\theta_n}},
$$
using \eqref{eq-5002} and \eqref{bruno}, we have
$$
\|\frac{(\partial_t-\partial^2_y)\partial_y u^n_{\theta_n}}
{\partial_y u^n_{\theta_n}}\|_{\mathcal{A}_{\ell}^{k}}
\le C_5 \theta_n^{\max\{0,\, k+2-\tilde{k}\}}.
$$
Plugging the above estimates into the definition of $\zeta^n$  in \eqref{zeta}, it leads
to the  \eqref{eq-zeta} given in Lemma \ref{lemm-zeta} and then completes its proof.

\smallskip
Finally, let us give the proof of Lemma \ref{lemm-e-1}.

\smallskip
\noindent\underline{\bf Proof of Lemma \ref{lemm-e-1}:}
Recall the definition
$$
e^{(1)}_n=\delta u^n\partial_x\delta u^n+\delta v^n\partial_y \delta u^n,
$$
and
$$
\begin{array}{ll}
e_{n}^{(2)}=& \big((1-S_{\theta_n}) (u^n-u^s)\big) \partial_x(\delta u^n)+
\delta u^n \partial_x \big((1-S_{\theta_n}) (u^n-u^s)\big)\\[2mm]
&+ \delta v^n \partial_y \big((1-S_{\theta_n}) (u^n-u^s)\big)
+\big((1-S_{\theta_n}) v^n\big) \partial_y(\delta u^n).
\end{array}
$$
We get
\begin{eqnarray}
\|\frac{e^{(1)}_j}{\partial_yu^s}\|_{\mathcal{A}_{\ell}^{k_1}}
&\le& M_{k_1}\Big\{\|\partial_x\delta u^j
\|_{\cA^{k_1}_{\ell}}\|\frac{\delta u^j}{\partial_yu^s}\|_{L^\infty}+\|\partial_x\delta u^j
\|_{L^2_{y, \ell}(L^\infty_{t, x})}\|\frac{\delta u^j}{\partial_yu^s}\|_{\mathcal{D}_{\ell}^{k_1}}\nonumber
\\
&&+\|\delta v^j\|_{L^\infty}
\|\frac{\partial_y\delta u^j}{\partial_yu^s}\|_{\mathcal{A}_{\ell}^{k_1}}
+\|\frac{\partial_y\delta u^j}{\partial_yu^s}\|_{L^2_{y,l}(L^\infty_{t,x})}
\| \delta v^j\|_{\cD^{k_1}_0}\Big\}.\label{esti-e-j-1}
\end{eqnarray}

Obviously, from
$$\frac{\partial_y\delta u^j}{\partial_yu^s}=
\frac{\partial^2_yu_{\theta_j}^j}{\partial_yu^s}\int_0^y w^j(t,x,\tilde{y})d\tilde{y}+
\frac{\partial_yu_{\theta_j}^j}{\partial_yu^s}w^j,$$
we have
\begin{equation}\label{A-k}
\begin{array}{lll}
&&\|\frac{\partial_y\delta u^j}{\partial_yu^s}\|_{\cA^{k_1}_{\ell}}
\le  \|\frac{\partial^2_y(u_{\theta_j}^j-u^s)}{\partial_yu^s}\|_{\cA^{k_1}_{\ell}}
\|w^j\|_{\cC^0_\ell}+\|\frac{\partial^2_yu^s}{\partial_yu^s}\|_{\cC^{k_1}_{\ell}}
\|w^j\|_{L^2_\ell}\\
&&+\|\frac{\partial^2_yu_{\theta_j}^j}{\partial_yu^s}\|_{\cC^{0}_{\ell}}
\|w^j\|_{\cA^{k_1}_\ell}+\|\frac{\partial_yu_{\theta_j}^j}{\partial_yu^s}\|_{L^\infty}
\|w^j\|_{\cA^{k_1}_{\ell}}+
\|\frac{\partial_yu_{\theta_j}^j}{\partial_yu^s}\|_{\dot{\cA}^{k_1}}
\|w^j\|_{L^\infty_{\ell}}\\[2mm]
&\le & C_0C_{11}\epsilon^2\theta_j^{\max\{3-\tilde{k}, k_1+5-2\tilde{k}\}}\Delta\theta_j+(C_{11}\epsilon+C^a)
C_0\epsilon\theta_j^{\max\{3-\tilde{k},\,k_1-\tilde{k}\}}
\Delta\theta_j\\[2mm]
& & + C^aC_0\epsilon \theta_j^{3-\tilde{k}}\Delta\theta_j
+2C_0\epsilon\theta_j^{\max\{3-\tilde{k},\,k_1-\tilde{k}\}}
\Delta\theta_j+(C_{11}\epsilon\theta_j^{\max\{0,\,k_1+1-\tilde{k}\}}
+C^a)
C_0\epsilon\theta_j^{3-\tilde{k}}\Delta\theta_j,
\end{array}
\end{equation}
where we have used  \eqref{eq-5002}, \eqref{bound} and \eqref{wn}.

If we choose a constant
$$\tilde{C}_{11}\ge C_0(2+3C^a+3C_{11}\epsilon_0),$$
then from \eqref{A-k}, we get
\begin{equation}\label{A-k-1}
\|\frac{\partial_y\delta u^j}{\partial_yu^s}\|_{\cA^{k_1}_{\ell}}\le \tilde{C}_{11}\epsilon
\theta_j^{\max\{3-\tilde{k}, k_1-\tilde{k}\}}\Delta\theta_j,
\end{equation}
by using $\tilde{k}\ge 7$.

By using \eqref{A-k-1}, Lemmas \ref{lemm-delta-u-j} and \ref{lemm-delta-v-j}, from \eqref{esti-e-j-1},
there exists a constant $C_8>0$ such that
$$
\|\frac{e^{(1)}_j}{\partial_yu^s}\|_{\mathcal{A}_{\ell}^{k_1}}
\le C_8\epsilon^2\theta_j^{\max\{6-2\tilde{k},\, k_1+3-2\tilde{k}\}}\Delta\theta_j,
$$
for all $k_1\le k_0-1$.

Similarly,
\begin{align*}
\|\frac{e^{(2)}_j}{\partial_y u^s}\|_{\mathcal{A}_{\ell}^{k_1}}&\le
\|\Big((1-S_{\theta_j})(u^j-u^s)\Big)
\frac{\delta u^j}{\partial_y u^s}
\|_{\mathcal{A}_{\ell}^{k_1+1}}\\
&+ \|\delta v^j \,\frac{\partial_y((1-S_{\theta_j})(u^j-u^s))}{\partial_y u^s}\|_{\mathcal{A}_{\ell}^{k_1}}
+\|\Big((1-S_{\theta_j})v^j\Big) \frac{\partial_y(\delta u^j)}{\partial_y u^s}
\|_{\mathcal{A}_{\ell}^{k_1}},
\end{align*}
by using Lemma \ref{lemm2.1}, the second formula of \eqref{f2.5} and \eqref{ineq},
we get
\begin{align*}
&\|\frac{e^{(2)}_j}{\partial_y u^s}\|_{\mathcal{A}_{\ell}^{k_1}}
\le M_{k_1}\Big\{
\|u^j-u^s\|_{\mathcal{A}_{\ell}^{k_1+1}}\|
\frac{\delta u^j}{\partial_y u^s}
\|_{L^\infty}\\
&+
\|(1-S_{\theta_j})(u^j-u^s)\|_{L^2_{y,\ell}(L^\infty_{t, x})}
\|\frac{\delta u^j}{\partial_y u^s}
\|_{{\cD}^{k_1+1}_{0}}
\\
&+ \|\delta v^j\|_{L^\infty} \,\|\frac{\partial_y(u^j-u^s)}{\partial_y u^s}\|_{\mathcal{A}_{\ell}^{k_1}}
+
\|\delta v^j\|_{{\mathcal{D}}_{0}^{k_1}} \,\|(1-S_{\theta_j})(\frac{\partial_y
(u^j-u^s)}{\partial_y u^s})\|_{L^2_{y,\ell }(L^\infty_{t, x})}\\
&+\|(1-S_{\theta_j})v^j\|_{L^\infty}\| \frac{\partial_y(\delta u^j)}{\partial_y u^s}
\|_{\mathcal{A}_{\ell}^{k_1}}
+\|v^j\|_{\mathcal{D}_{0}^{k_1}}\,\|
\frac{\partial_y(\delta u^j)}{\partial_y u^s}
\|_{L^2_{y,\ell }(L^\infty_{t, x})}\Big\}\\
&\le M_{k_1}\Big\{
\|u^j-u^s\|_{\mathcal{A}_{\ell}^{k_1+1}}\|
\frac{\delta u^j}{\partial_y u^s}
\|_{L^\infty}+
\theta_j^{2-k'}\|u^j-u^s\|_{\cA^{k'}_{\ell}}
\|\frac{\delta u^j}{\partial_y u^s}
\|_{{\cD}^{k_1+1}_{0}}
\\
&+ \|\delta v^j\|_{L^\infty} \,\|\frac{\partial_y(u^j-u^s)}{\partial_y u^s}\|_{\mathcal{A}_{\ell}^{k_1}}
+
\theta_j^{2-k'}\|\delta v^j\|_{{\mathcal{D}}_{0}^{k_1}} \,\|\frac{\partial_y
(u^j-u^s)}{\partial_y u^s}\|_{\mathcal{A}_{\ell}^{k'}}\\
&+\theta_j^{2-k'}\|v^j\|_{\mathcal{D}_{0}^{k'}}\| \frac{\partial_y(\delta u^j)}{\partial_y u^s}
\|_{\mathcal{A}_{\ell}^{k_1}}
+\|v^j\|_{\mathcal{D}_{0}^{k_1}}\,\|
\frac{\partial_y(\delta u^j)}{\partial_y u^s}
\|_{L^2_{y,\ell }(L^\infty_{t, x})}\Big\},\end{align*}
for a fixed integer  $2\le k'\le k_0-2$.

By applying Lemma \ref{lemm-delta-u-j}, Lemma \ref{lemm-delta-v-j} and estimates
\eqref{eq-5002}, \eqref{A-k-1} to the above inequality, and by setting $k'\ge \tilde{k}-2$,
it follows that for $k_1\le k_0-1$,
$$
\|\frac{e^{(2)}_j}{\partial_y u^s}\|_{\mathcal{A}_{l}^{k_1}}
\le C_8
\epsilon^2 \theta_j^{\max(3-\tilde{k},k_1+5-2\tilde{k})}\Delta \theta_j,
$$
for a positive constant $C_8$. And this completes the proof of the lemma.

\bigskip
\noindent
{\bf Acknowledgements :}
The research of the first author was supported in part by the Zhiyuan
foundation, Shanghai Jiao Tong University and the National
Science Foundation of China No. 11171211.
The research of the second author was supported in part by the National
Science Foundation of China Nos. 10971134 and 11031001.  The
research of the third author was supported partially
by the Fundamental Research Funds for the Central Universities and
the National Science Foundation of China No. 11171261. The last
author's research was supported by the General Research Fund of Hong Kong,
CityU No. 104511, and the Croucher Foundation. The authors would like to express their gratitude to Weinan E
and Zhouping Xin for their valuable discussions.


\smallskip


\begin{thebibliography}{99}

\bibitem{ali-gerard} Alinhac, S. \& G\'erard, P.: {\it Pseudo-Differential Operators
and The Nash-Moser Theorem}. Translated from the 1991 French original by Stephen S. Wilson.
Graduate Studies in Mathematics, 82. American Mathematical Society, Providence, RI, 2007. viii+168 pp.

\bibitem{caf} Caflisch, R.E. \& Sammartino, M.: Existence and singularities for the
Prandtl boundary layer equations, {\it Z. Angew. Math. Mech.}, 80(2000), 733-744.
\bibitem{Piriou} Chazarain, J.  \& Piriou, A.: {\it Introduction to The Theory of
Linear Partial Differential Equations}, North-Holland
Publishing Company, Amsterdam, 1982.

\bibitem{e-1} E, W.: Boundary layer theory and the zero-viscosity limit of
the Navier-Stokes equation. {\it Acta Math. Sin. (Engl.
Ser.)} 16(2000), 207-218.

\bibitem{e-2} E, W. \&  Engquist, B.: Blow up of solutions of the unsteady
Prandtl's equation, {\it Comm. Pure Appl.
Math.}, 50(1997), 1287-1293.

\bibitem{GV-D} G\'erard-Varet, D. \& Dormy, E.: On the ill-posedness of the Prandtl equation,
{\it J. Amer. Math. Soc.},  23(2010), 591-609.

\bibitem{GV-N} G\'erard-Varet, D. \& Nguyen, T.: Remarks on the ill-posedness of the Prandtl equation,
preprint 2010. arXiv:1008.0532v1.

\bibitem{grenier} Grenier, E.: On the nonlinear instability of Euler and Prandtl equations.
{\it Comm. Pure Appl. Math.} 53(2000), 1067-1091.

\bibitem{guo} Guo, Y. \& Nguyen, T.: A note on the Prandtl boundary layers, {\it Comm. Pure Appl. Math.}
 64 (2011) 1416-1438, doi: 10.1002/cpa.20377

\bibitem{hami} Hamilton, R.: The inverse function theorem of Nash and Moser, {\it Bull. Amer.
Math. Soc.}, 7(1982), 65-222.

\bibitem{hong-hunter} Hong, L. \& Hunter, J. K.: Singularity formation and instability in the unsteady
inviscid and viscous Prandtl equations, {\it Commun. Math. Sci.} 1 (2003), 293-316.

\bibitem{hormander} H\"{o}rmander, L.: {\it The Analysis of Linear
Partial Differential Operators} (Vol. I-IV), Springer Verlag 1985.

\bibitem{hor} { H\"ormander, L.:} The boundary problems of physical geodesy.
{\it Arch. Rational Mech. Anal.}, 62(1982), 1-52.

\bibitem{cannone} Lombardo, M. C., Cannone, M. \& Sammartino, M.: Well-posedness of the boundary layer
equations. {\it SIAM J. Math. Anal.}, 35(2003), 987-1004 (electronic).

\bibitem{metivier} M\'etivier, G.: {\it Small Viscosity and Boundary Layer Methods.
Theory, Stability Analysis, and Applications}. Modeling and Simulation in Science,
Engineering and Technology. Birkhauser Boston, Inc., Boston, MA, 2004. xxii+194 pp.

\bibitem{moser} Moser, J.:
A new technique for the construction of solutions of nonlinear differential
equations, {\it Proc. Nat. Acad. Sci.}, 47(1961), 1824-1831.

\bibitem{nash} Nash, J.:
The imbedding problem for Riemannian manifolds, {\it Ann. of Math.}, 63(1956), 20-63.
\bibitem{oleinik-1} Oleinik, O. A.: The Prandtl system of equations in boundary layer theory.
{\it Soviet Math Dokl}, 4(1963), 583-586.

\bibitem{oleinik-3} {\sc Oleinik, O. A. \& Samokhin, V. N.:} {\it Mathematical Models in Boundary
Layers Theory}. Chapman \& Hall/CRC, 1999.

\bibitem{prandtl} {\sc Prandtl, L.:} \"{U}ber Fl\"ussigkeitsbewegungen bei sehr kleiner
Reibung. In ``{\it Verh. Int. Math. Kongr., Heidelberg 1904"},
Teubner 1905, 484-494.

\bibitem{Samm} {\sc Sammartino, M. \& Caflisch, R. E.:} Zero viscosity limit
for analytic solutions of the Navier-Stokes equations on a
half-space, I. Existence for Euler and Prandtl equations. {\it Comm.
Math. Phys.}, 192(1998), 433-461; II. Construction of the
Navier-Stokes solution. {\it Comm. Math. Phys.}, 192(1998), 463-491.

\bibitem{xin-zhang} Xin, Z. \& Zhang, L.: On the global existence
of solutions to the Prandtl's system, {\it Adv. Math.}, 181(2004), 88-133.

\end{thebibliography}
\end{document}